\documentclass[12pt, leqno]{article}
\usepackage{amsmath}
\usepackage{amsthm}
\usepackage{amssymb}
\usepackage{latexsym}
\allowdisplaybreaks[1]
\usepackage{amsfonts}
\usepackage{ascmac}
\usepackage{color}
\usepackage{enumerate}
\usepackage{graphicx}
\usepackage{bm}
\usepackage{mathrsfs}

\theoremstyle{definition}
\theoremstyle{plain}
\allowdisplaybreaks[1]
\date{}
\newtheorem{Thm}{Theorem}[section]
\newtheorem{Prop}[Thm]{Proposition}
\newtheorem{Lemma}[Thm]{Lemma}
\newtheorem{Cor}[Thm]{Corollary}

\newcommand{\p}{\partial}

\newcommand{\dis}{\displaystyle}

\newcommand{\norm}{\parallel}

\newcommand{\Z}{{\mathbb Z}}

\newcommand{\N}{{\mathbb N}}
\newcommand{\R}{{\mathbb R}}

\newcommand{\ep}{\varepsilon }

\newcommand{\Omegain}{ {\Omega_h\setminus\p\Omega_h}}
\def\text#1{\mbox{#1 }}

\title{\bf Finite difference methods for \\ linear transport equations}
\author{Kohei Soga
\footnote{Department of Mathematics, Faculty of Science and Technology, Keio University, 3-14-1 Hiyoshi, Kohoku-ku, Yokohama, 223-8522, Japan. E-mail:  soga@math.keio.ac.jp 
}
}
\begin{document}
\maketitle
\begin{abstract} 
\noindent DiPerna-Lions (Invent.\,Math.,\,1989) established the existence and uniqueness results for linear transport equations with Sobolev velocity fields. 
This paper provides mathematical analysis on two simple finite difference methods applied to linear transport equations on a bounded domain with divergence-free (unbounded) Sobolev velocity fields.  
The first method is based on a Lax-Friedrichs type explicit scheme with a generalized hyperbolic scale, where  truncation of an unbounded velocity field and its measure estimate are implemented to ensure the monotonicity of the scheme; the method is $L^p$-strongly convergent. 
The second method is based on an implicit scheme with $L^2$-estimates, where the discrete Helmholtz-Hodge decomposition for discretized velocity fields plays an important role to ensure the divergence-free constraint in the discrete problem; the method is scale-free and $L^2$-strongly convergent. 
The key point for both of our methods is to obtain fine $L^2$-bounds of approximate solutions that tend to the norm of the exact solution given by DiPerna-Lions. 
Finally, the explicit scheme is applied to the case with smooth velocity fields from the viewpoint of  the level-set method involving transport equations, where rigorous  discrete approximation of geometric quantities of  level sets is discussed.      
 
\medskip\medskip

\noindent{\bf Keywords:} transport equation, DiPerna-Lions theory, finite difference method, level-set method 
\medskip

\noindent{\bf AMS subject classifications:} 35Q49,  35D30,  65M06, 65M12, 53A05

\end{abstract}
%

\setcounter{section}{0}
\setcounter{equation}{0}
\section{Introduction}
We consider linear transport equations
\begin{eqnarray}\label{T.eq} 
&&\begin{cases}\dis
\p_t f(t,x)+v(t,x)\cdot \nabla f(t,x)=0\quad \rm{in}\,\,\,\, (0,T]\times \Omega,\\ 
f(0,\cdot)= f^0\quad \rm{on}\,\,\,\,\Omega,  
\end{cases}\\\nonumber
&&\Omega\subset\R^3\mbox{ is a bounded connected open set},\\\nonumber 
&&v\in L^2([0,T];H^1_0(\Omega)^3)\cap L^\infty([0,T];L^2(\Omega)^3) \mbox{ with $\nabla\cdot v=0$ is a given  function},\\\nonumber 
&&f^0\in L^\infty(\Omega)\mbox{ or } L^2(\Omega) \,\,\mbox{ is initial data},
\end{eqnarray}
where $T>0$ is an arbitrary terminal time, $f:[0,T]\times\Omega\to\R$ is the unknown function, $\p_t,\p_{x_j}$ etc. stand for the (weak) differentiations, $\nabla=(\p_{x_1},\p_{x_2},\p_{x_3})$,  $L^p(\Omega)=L^p(\Omega;\R)$, $H^1_0(\Omega)=H^1_0(\Omega;\R)=\overline{C^\infty_0(\Omega;\R)}^{\norm\cdot\norm_{H^1}}$ and $x\cdot y=\sum_{i=1}^3x_iy_i$ for $x,y\in\R^3$. 

The problem \eqref{T.eq} appears in classical fluid mechanics together with the ODE
\begin{eqnarray}\label{ODE}
x'(s)=v(s,x(s)),
\end{eqnarray}
where $v$ is an Eulerian velocity field of an incompressible fluid flow and the flow of \eqref{ODE} gives trajectories of Lagrangian fluid particles. When $v$ is $C^1$-smooth, the method of characteristics connects  \eqref{T.eq} and \eqref{ODE}, providing full information on both problems. However, it is not clear if one can always find a smooth velocity field of a fluid flow. DiPerna-Lions \cite{DiPerna-Lions} generalized the classical theory to the problems  \eqref{T.eq} and \eqref{ODE} with velocity fields in Sobolev spaces (under weaker assumptions than ours), introducing weak solutions of \eqref{T.eq}  and a notion of the ``flow'' of \eqref{ODE}. We refer to \cite{AC} for further development of DiPerna-Lions theory.         

The regularity class of the velocity field mentioned in  \eqref{T.eq} is motivated by the class of Leray-Hopf weak solutions to the  incompressible homogeneous Navier-Stokes equations, the class of weak solutions of  incompressible inhomogeneous Navier-Stokes equations studied in \cite{Simon}  or the class of generalized solutions to the incompressible two-phase Navier-Stokes equations introduced in \cite{Abels} (so-called varifold solutions). In an  inhomogeneous flow, the density profile is transported by the velocity field, and the incompressible Navier-Stokes equations are coupled with the linear transport equation. In a two-phase flow, an  interface is transported by the velocity field. The  level-set method is an effective way to chase   moving interfaces (see \cite{SS}), where level-set functions are given as solutions of the linear transport equation. 
Our motivation for  \eqref{T.eq}  comes from mathematical/numerical analysis on incompressible inhomogeneous flows and  two-phase flows with the level-set method. In particular, {\it we focus our attention on the theoretical aspect of classical finite difference methods of transport equations for possible applications in fluid mechanics.} We refer to \cite{G} and the references therein for physics of the miscible mixing of two incompressible fluids (inhomogeneous flows); we refer also to \cite{BR} for recent developments of theory of two-phase flows in the framework of computational fluid mechanics. We remark that, in the literature, numerical analysis on continuity equations has been more frequent in the framework of finite volume methods, particularly in the context of compressible fluid mechanics.             

Before going into the main results of this paper, we briefly explain how to verify the existence and uniqueness for  \eqref{T.eq} by means of DiPerna-Lions' existence and   uniqueness results obtained for the whole space problems. This is possible in the case where $v$ vanishes on $\p\Omega$. Note that, in the case where $v$ does not vanish on $\p\Omega$, existence and uniqueness  must be analyzed independently: see \cite{Boyer0}.   A function $f$ is a weak solution of \eqref{T.eq} with $f^0\in L^\infty(\Omega)$ (resp. $f^0\in L^2(\Omega)$), if $f\in L^\infty([0,T];L^\infty(\Omega))$  (resp. $f\in L^\infty([0,T];L^2(\Omega))$) and 
\begin{eqnarray}\label{weak-sol}
&&\int_0^T\int_\Omega\Big( f(t,x)\p_t\varphi(t,x) +f(t,x)v(t,x)\cdot\nabla \varphi(t,x)\Big)\,dxdt+\int_\Omega f^0(x)\varphi(0,x)dx=0\\\nonumber
&&\mbox{for all  $\varphi\in C^\infty([0,T]\times\R^3;\R)$ with supp$(\varphi)\subset[0,T)\times\R^3$ compact}.
\end{eqnarray}
Note that test functions are NOT taken as compactly supported elements of  $C^\infty([0,T]\times\Omega;\R^3)$. Our choice of the test functions implies equivalence between \eqref{T.eq} and the whole space problem 
\begin{eqnarray}\label{T2.eq} 
&&\begin{cases}\dis
\p_t \tilde{f}(t,x)+\tilde{v}(t,x)\cdot \nabla \tilde{f}(t,x)=0\quad \rm{in}\,\,\,\, (0,T]\times \R^3,\\ 
\tilde{f}(0,\cdot)= \tilde{f}^0\quad \rm{on}\,\,\,\,\R^3,  
\end{cases}\\\nonumber
&&\mbox{$\tilde{v}$  is the $0$-extension of $v$ to $[0,T]\times\R^3$; $\tilde{f}^0$ is  the $0$-extension of $f^0$ to $\R^3$}, 
\end{eqnarray}
where $\tilde{v}\in  L^2([0,T];H^1_0(\R)^3)\cap L^\infty([0,T];L^2(\R)^3)$ with $\nabla\cdot \tilde{v}=0$. 
A function $\tilde{f}$ is a weak solution of \eqref{T2.eq} with $ \tilde{f}^0\in L^\infty(\R^3)$ (resp. $\tilde{f}^0\in L^2(\R^3)$), if $\tilde{f}\in L^\infty([0,T];L^\infty(\R^3))$ (resp. $\tilde{f}\in L^\infty([0,T];L^2(\R^3))$) and 
\begin{eqnarray}\label{weak-sol2}
&&\int_0^T\int_{\R^3}\Big( \tilde{f}(t,x)\p_t\varphi(t,x) +\tilde{f}(t,x)v(t,x)\cdot\nabla \varphi(t,x)\Big)\,dxdt+\int_{\R^3} \tilde{f}^0(x)\varphi(0,x)dx=0\\\nonumber
&&\mbox{for all  $\varphi\in C^\infty([0,T]\times\R^3;\R)$ with supp$(\varphi)\subset[0,T)\times\R^3$ compact}.
\end{eqnarray}
It is clear that the $0$-extension of  a weak solution of  \eqref{T.eq} to $[0,T]\times\R^3$ is a weak solution of \eqref{T2.eq}. Conversely, let $\tilde{f}$ be a weak solution of \eqref{T2.eq}. It follows from \eqref{weak-sol2} that we have for any test function of the form $\varphi(t,x)=\varphi_1(t)\varphi_2(x)$ with supp$(\varphi_1)\subset [0,T)$ and supp$(\varphi_2)\subset \R^3\setminus  \bar{\Omega}$,
$$\int_{\R^3\setminus\bar{\Omega}}\Big(\int_0^T\tilde{f}(t,x)\varphi_1'(t)dt\Big)\varphi_2(x)dx=0,$$
which implies 
$$\int_0^T\tilde{f}(t,x)\varphi_1'(t)dt=0,\quad {\rm a.e.} \,\,\,x\in \R^3\setminus\bar{\Omega}.$$
Hence, we see that $\tilde{f}(t,x)$ is independent of $t$ for a.e. $x\in \R^3\setminus\bar{\Omega}$ fixed, from which we obtain    
$$-\varphi_1(0)\int_{\R^3\setminus\bar{\Omega}}\tilde{f}(t,x)\varphi_2(x)\,dx=0.$$ 
Therefore, we conclude that $\tilde{f}(t, x)=0$ a.e. $(t,x)\in[0,T]\times( \R^3\setminus\bar{\Omega})$. Now it is clear that $f:=\tilde{f}|_{x\in\Omega}$ is a weak solution of \eqref{T.eq}. 

In our upcoming investigation, it is convenient to have quantities that vanish near the boundary due to the use of  the (discrete) integration by parts.    For this reason, we re-set the transport equation in a bounded open set $\tilde{\Omega}\subset\R^3$ containing $\bar{\Omega}$ with the $0$-extension of $v$ to $[0,T]\times\tilde{\Omega}$ and the $0$-extension of $f^0$ to $\tilde{\Omega}$.     
To sum up, we reach the following important facts: 
{\it 
\begin{enumerate}
\item[(F1)] Due to DiPerna-Lions  theory \cite{DiPerna-Lions}, there exists the unique weak solution $f$ of \eqref{T.eq} and $\tilde{f}$ of \eqref{T2.eq}, which satisfy  $\tilde{f}|_{x\in\Omega}=f$ and  
$$\norm \tilde{f}(t,\cdot)\norm_{L^2(\R^3)} = \norm f(t,\cdot)\norm_{L^2(\Omega)}=\norm \tilde{f}^0\norm_{L^2(\R^3)}=\norm f^0\norm_{L^2(\Omega)},  \mbox{\quad  a.e. $t\in[0,T]$}$$
 (see equality (26) of \cite{DiPerna-Lions}), where the $L^\infty_t L^\infty_x$-weak solution is regarded as an  $L^\infty_t L^2_x$-weak solution. 
\item[(F2)] Without loss of generality, we may investigate \eqref{T.eq} assuming in addition that {\rm supp}$(v)\subset[0,T]\times \Omega$ and  {\rm supp}$(f^0)\subset\Omega$. 
\end{enumerate}
}

This paper provides mathematical analysis on a two simple finite difference methods of \eqref{T.eq}. The novelty is that {\it this is the first attempt to apply  classical finite difference methods to DiPerna-Lions theory; a new idea is given to deal with explicit schemes without assuming the space-time $L^\infty$-bound of the velocity field.}  

In Section 3, we discuss  the first method, where we solve  \eqref{T.eq} with $f^0\in L^\infty(\Omega)$ based on a Lax-Friedrichs type explicit scheme under a generalized hyperbolic scaling condition for the space-time mesh size. 
To ensure the monotonicity of the scheme (i.e., the CFL-condition), we implement truncation of a possibly unbounded velocity field in terms of the mesh size  together with a suitable measure estimate for the truncated part (its measure must be arbitrarily small  as the mesh size tends to $0$). The measure estimate requires the embedding  $L^2([0,T];H^1_0(\Omega)^3)\cap L^\infty([0,T];L^2(\Omega)^3)\subset L^3([0,T];L^3(\Omega)^3)$. The monotonicity implies the comparison principle for the discrete problem, which would be important also in the context of the level-set method. Our  explicit method is $L^p$-strongly convergent for all $1\le  p<\infty$ (this is proven via the $L^2$-strong convergence). We remark that our generalized hyperbolic scaling condition would not be so convenient in actual computations of hyperbolic problems due to stronger numerical diffusiveness; on the other hand, our idea would be effectively applied to possible explicit schemes for (quasilinear) parabolic problems with an unbounded advection coefficient.           

In Section 4, we discuss the second method, where we  solve \eqref{T.eq} with $f^0\in L^2(\Omega)$ based on an implicit scheme and its $L^2$-estimates. A formal calculation for $\eqref{T.eq}$ with the divergence-free constraint of $v$ implies that the $L^2$-norm of $f(t,\cdot)$ does not depend on $t$, which is the key observation to obtain $L^2$-weak solutions. We follows the same procedure for the discrete problem with the discrete divergence-free constraint.  Since the discretized velocity field is not necessarily (discrete)  divergence-free, we take out  its divergence-free part by means of the discrete Helmholtz-Hodge decomposition. The whole discussion is compatible with the fully discrete finite difference projection methods originally introduced in  \cite{Chorin} and further developed in \cite{Kuroki-Soga} and \cite{Maeda-Soga}.    
Our implicit method is scale-free and $L^2$-strongly convergent. 

The structure of the transport equation does not obviously allow a priori estimates of the derivatives of solutions. Hence, weak convergence seems to be maximal in our methods. 
However,  due to fine $L^2$-bounds of approximate solutions that tend to the norm of the exact solution given in  DiPerna-Lions theory, we find that the weak convergence is in fact strong convergence. It is interesting to compare our strong convergence results with the work \cite{BJ}, in which the authors proved that an abstract explicit numerical scheme on the Cartesian grid applied to nonlinear continuity equations (linear transport equations are included) is $L^1_{\rm loc}$-strongly convergent, provided velocity fields have the space-time $L^\infty$-bound so that the scheme is monotone. Their argument is based on a priori estimates in an intermediate space between Sobolev spaces and Lebesgue spaces which is still compactly embedded in a Lebesgue space, i.e., solutions of (nonlinear) continuity equations have regularity weaker than Sobolev but better than Lebesgue.  Finally, we refer to the work \cite{Boyer} and the references therein for stability and convergence results of the implicit upwind finite volume method applied to linear continuity equations, where the author proved strong convergence in $C^0([0,T];L^p(\Omega))$.  The time-uniformity  was indirectly derived by comparing the discrete problem, the exact problem and the regularized problem with the commutator error term. The method given in  \cite{Boyer}  has been further investigated to obtain rates of weak convergence: see \cite{Schlichting} and the references therein.  
Although we do not go into detail in this paper, one could apply the same  procedure as Section 5 of \cite{Boyer} to our finite difference methods for possible strong convergence in $C^0([0,T];L^p(\Omega))$. A rate of (weak)  convergence of our finite difference methods is a completely open question.  Some part of our results would still hold for a problem having a velocity field without $W^{1,p}$-regularity,   since the finite difference methods do not need the trace of the velocity field to define its discretization (i.e.,  a finite volume method requires the trace of a given velocity field on the boundary of each mesh to define its discretization).        

In Section 5, we apply our explicit scheme to the problem with a smooth velocity field from the viewpoint of the level-set method, and demonstrate error estimates up to the second order $x$-derivatives. Then, we discuss  discrete approximation of geometric quantities such as the unit normal vector field, the mean curvature, the area element, etc., of a level-set given by a smooth solution of \eqref{T.eq}. 
\setcounter{section}{1}
\setcounter{equation}{0}
\section{Preliminary}

Let $\tau,h>0$ be the mesh size for time, space, respectively. We will refer to the scale condition later. 
The discrete time is defined as $t_n:=n\tau$ for $n\in\N\cup\{0\}$. Let $T_\tau\in\N$ be such that $T\in[\tau T_\tau,\tau T_\tau+\tau)$. $T_\tau$ is the terminal time of the discrete problems. 
We introduce the grid $h\Z^3:=\{ (hz_1,hz_2,hz_3)\,|\,z_1,z_2,z_3\in\Z \}$. Let $e^1,e^2,e^3$ be the standard basis of $\R^3$. 
The boundary of  a set $A\subset h\Z^3$ is defined as $\partial A:=\{ x\in A\,|\,\{x\pm he^i\,|\,i=1,2,3\}\not\subset A \}$. 

Let $\Omega$ be a bounded connected open subset of $\R^3$. Set 
$$C_h(x):=\Big[x_1-\frac{h}{2},x_1+\frac{h}{2}\Big)\times\Big[x_2-\frac{h}{2},x_2+\frac{h}{2}\Big)\times\Big[x_3-\frac{h}{2},x_3+\frac{h}{2}\Big).$$
The discretization of $\Omega$ is defined as  
$$\Omega_h:=\{x\in \Omega\cap h\Z^3\,|\, C_{2h}(x)\subset\Omega\}.$$ 
Define the discrete $x$-derivatives of a function $\phi:\Omega_h\to\R$ as
\begin{eqnarray*}
&&D_i^+\phi(x):=\frac{\phi(x+he^i)-\phi(x)}{h},\,\,\,D_i^-\phi(x):=\frac{\phi(x)-\phi(x-he^i)}{h},\\ 
&&D_i\phi(x):=\frac{\phi(x+he^i)-\phi(x-he^i)}{2h},\,\,\,
D_i^2\phi(x):=\frac{\phi(x+he^i)+\phi(x-he^i)-2\phi(x)}{h^2},\\
&&D^\pm:=(D^\pm_1,D^\pm_2,D^\pm_3),\,\,\,D:=(D_1,D_2,D_3)\quad\mbox{ for each $x\in \Omega_h$},
\end{eqnarray*}  
where {\it we always assume that $\phi$ is extended to be $0$ outside $\Omega_h$, i.e., $\phi(x+ he^i)=0$ (resp. $\phi(x- he^i)=0$) in the above definition if $x+ he^i\not\in \Omega_h$ (resp. $x- he^i\not\in \Omega_h$)}. 
 We often use the summation by parts such as  
\begin{eqnarray}\label{by-parts}
&&\sum_{x\in \Omega_h\setminus \partial \Omega_h}w(x) D_i^+\phi(x)=-\sum_{x\in \Omega_h\setminus \partial\Omega_h}D^-_i w(x) \phi(x),\\\nonumber
&&\sum_{x\in \Omega_h\setminus \partial \Omega_h}w(x) D_i\phi(x)=-\sum_{x\in \Omega_h\setminus \partial\Omega_h}D_i w(x) \phi(x)
\end{eqnarray}
for functions $w,\phi:\Omega_h\to\R$ with $w|_{\p\Omega}=\phi|_{\partial \Omega_h}=0$.  

We define the discrete $L^p$-norms of a function $\phi:A\to\R$ or $\phi:A\to\R^3$  with $A\subset h\Z^3$ as 
\begin{eqnarray*}
\norm \phi\norm_{p,A}:=\Big(\sum_{x\in A}|\phi(x)|^ph^3 \Big)^{\frac{1}{p}},\,\,\,\,\norm \phi\norm_{\infty,A}:=\max_{x\in A}|\phi(x)| ;
\end{eqnarray*}
in particular for $p=2$, we introduce the discrete inner product as  
\begin{eqnarray*}
(\phi,\tilde{\phi})_{2, A}:=\sum_{x\in A}\phi(x)\tilde{\phi}(x)h^3,\quad \norm \phi\norm_{2,A}=\sqrt{(\phi,\phi)_{2,A}}. 
\end{eqnarray*}


\setcounter{section}{2}
\setcounter{equation}{0}
\section{Explicit method}

In this section, we investigate an explicit finite difference  method of \eqref{T.eq}. 
Our explicit scheme is formulated on the whole grid  $h\Z^3$; namely,  \eqref{T2.eq} is discretized. {\it The scheme is ``local'' in the sense that the presence of $h\Z^3\setminus\Omega_h$ does not change point-wise  features}; for actual computation, we may look at the values only on $\Omega_h$, even though the support of a solution slightly spreads outside $\Omega_h$ due to numerical viscosity. It is possible to formulate an explicit scheme only on $\Omega_h$ with the artificial boundary condition on $\p\Omega_h$, at least for the construction of a weak solution. However, the presence of the boundary condition is not convenient for error analysis of the derivatives in the smooth case discussed in Section 5. 

Consider initial data $f^0\in L^\infty(\Omega)$ such that $f^0\equiv 0$ in a neighborhood of $\p\Omega$. Note that $f^0\in L^p(\Omega)$ for all $p\in[1,\infty]$, since $\Omega$ is bounded. 
We extend $f^0$ to be $0$ for all $x\in\R^3\setminus \Omega$ and write it as $\tilde{f}^0$.  Define $g^0:h\Z^3\to\R$ as 
\begin{eqnarray}\label{3ini}
g^0(x):=\frac{1}{ h^3}\int_{C_h(x)}\tilde{f}^0(y)\,dy.
\end{eqnarray}
Fix $v\in L^2([0,T];H^1_0(\Omega)^3)\cap L^\infty([0,T];L^2(\Omega)^3)$  such that $\nabla\cdot v=0$  and supp$(v)\subset[0,T]\times \Omega$. 
We extend $v$ to be $0$ for all $x\in\R^3\setminus \Omega$ and write it as $\tilde{v}$.  
Define $u^n=(u^n_1,u^n_2,u^n_3):h\Z^3\to\R^3$, $n=0,1,\ldots, T_\tau-1$  as 
$$u^n_i(x):=\frac{1}{\tau h^3}\int_{\tau n}^{\tau(n+1)}\int_{C_h(x)}\tilde{v}_i(s,y)\,dyds. $$ 
Note that $g^0\equiv 0$, $u^n\equiv0$ near the boundary of $\Omega_h$ and outside $\Omega_h$ for all sufficiently small $h>0$. We always consider such $h>0$.         
\subsection{Basic idea}
We want to design a discrete problem given in an explicit way so that the problem satisfies monotonicity. Naively, one would think of a Lax-Friedrichs type method 
\begin{eqnarray}\label{naive} 
&&\Big(g^{n+1}(x)-  \frac{1}{7}\sum_{\omega\in B}g^n(x+h \omega)  \Big)\frac{1}{\tau} +\sum_{j=1}^3u^n_j(x) D_jg^n(x)=0 \quad \mbox{for \,\,} x\in h\Z^3,\\\nonumber
&&B:=\{\pm e^1,\pm e^2,\pm e^3,0\}, \quad \mbox{$g^0$ is given as \eqref{3ini}},
\end{eqnarray}
which can be rewritten as 
\begin{eqnarray}\label{naive2} 
g^{n+1}(x) &=&  \frac{1}{7}g^n(x)
+\sum_{j=1}^3\Big\{\Big(\frac{1}{7}+\frac{\tau}{2h}u^n_j(x)\Big)g^n(x-he^j)\\\nonumber
&& +\Big(\frac{1}{7}-\frac{\tau}{2h}u^n_j(x)\Big)g^n(x+he^j) \Big\}.
\end{eqnarray}
In this formulation, we need to ensure the CFL-type condition 
\begin{eqnarray}\label{3CFL}
\frac{1}{7}\pm\frac{\tau}{2h}u^n_j(x)\ge0
\end{eqnarray} 
and the (generalized) hyperbolic-type scaling 
$$(h,\tau)\to0\mbox{ \,\,\, in such a way that \,\,\,}\frac{h^2}{\tau}\to0.$$
Note that, without the CFL-condition, we are not able to obtain monotonicity nor convergence of \eqref{naive} in general; if $v\in L^\infty$, one can take the standard hyperbolic scaling $\tau=O(h)$; without  the hyperbolic-type scaling,  
\eqref{naive} would produce a solution of a parabolic transport equation (see \cite{Oleinik}); one could think of a mixture of upwind/downwind type methods (i.e., one uses $D^\pm_j$ in \eqref{naive} instead of $D_j$ depending on the sign of $u^n_j(x)$) in order to be free from the scaling conditions, but the presence of both $D^\pm_j$ causes a trouble in \eqref{by-parts} failing to obtain the weak form.   

The major obstacle is that  $|u^n|$ is allowed to be unbounded.  Hence,  we need a truncation technique, i.e.,  instead of $u^n$ in \eqref{naive}, we introduce $\tilde{u}^n=(\tilde{u}^n_1,\tilde{u}^n_2,\tilde{u}^n_3)$ as  
\begin{eqnarray}\label{truncation}
\tilde{u}^n_j(x):=
\begin{cases}\medskip
u^n_j(x),\mbox{\,\,\, if $|u^n_j(x)|\le h^{-\beta}$},\\ 
{\rm sign}(u^n_j(x))h^{-\beta},\mbox{\,\,\, if $|u^n_j(x)|> h^{-\beta}$}
\end{cases},\quad \beta>0
\end{eqnarray}  
 and take a scaling 
$$ \tau=O(h^{2-\alpha}),\quad \alpha>0$$
with an appropriate  choice of positive constants $\alpha,\beta$.  Note that {\it $\tilde{v}|_{\R^3\setminus\Omega}\equiv0$, $u^n|_{h\Z^3\setminus\Omega_h}\equiv0$ and the truncation takes place only within $\Omega_h$.}          
 Then, we need careful treatment of the truncated parts, in particular, for the term $D\cdot \tilde{u}^n(x)$ appearing in the weak form. This  will be done by a measure estimate of  the truncated points: Set  
$$A^n_{h,j}:=\{x\in\Omega_h\,|\,\tilde{u}^n_j(x)\neq u^n_j(x)\},\,\,\,{\rm vol}(A^n_{h,j}):=\sum_{x\in A^n_{h,j}}h^3.$$
The $L^2$-bound of $v(t,\cdot)$ implies  the $\norm\cdot\norm_{2,\Omega_h}$-bound of $u^n$, from which we have  
$$\sum_{n=0}^{T_\tau-1}{\rm vol}(A^n_{h,j})\tau \le O(h^{2\beta}),$$
but unfortunately this is not sharp enough as we will see later. Hence, we use the fact  
$$L^2([0,T];H^1_0(\Omega))\cap L^\infty([0,T];L^2(\Omega)) \subset L^3([0,T];L^3(\Omega))$$
to obtain 
$$\sum_{n=0}^{T_\tau-1}{\rm vol}(A^n_{h,j})\tau \le O(h^{3\beta}).$$
\indent As we will see later, for the solution $g^n$ of \eqref{naive} ($u$ must be replaced by $\tilde{u}$),  $\norm g^n\norm_{\infty,h\Z^3}$ is  bounded by $\norm f^0\norm_{L^\infty(\Omega)}$ and $\norm g^n\norm_{p,h\Z^3}$ is bounded by $\norm f^0\norm_{L^p(\Omega)}+\mbox{ [small error]}$ for all $1\le p<\infty$. The latter bound is due to the fact that  the sum of the coefficients of $g^n(x+h \omega)$, $\omega \in B$ in \eqref{naive2} is equal to $1$.  This is another virtue of the non-divergence form of \eqref{naive}. 
\subsection{Discrete problem (explicit)}

We describe our discrete problem with unknowns $g^1,\ldots,g^{T_\tau}:h\Z^3\to\R$ as follows:   
\begin{eqnarray}\label{scale}
&&\mbox{$\tau=h^{2-\alpha}$,\,\,\, $\dis h^{1-(\alpha+\beta)}\le \frac{2}{7}$, \,\,\,  $\alpha>0$, \,\,\,$\dis \beta>\frac{1}{2}$, $\alpha+\beta<1$, }\\\label{explicit1}
&& \Big(g^{n+1}(x)-  \frac{1}{7}\sum_{\omega\in B}g^n(x+h \omega)  \Big)\frac{1}{\tau} +\sum_{j=1}^3\tilde{u}^n_j(x) D_jg^n(x)=0,\,\,\,x\in h\Z^3,
\end{eqnarray}
where $\tilde{u}^n$ is defined as \eqref{truncation},  $g^0$ is given as \eqref{3ini} and $B=\{\pm e^1,\pm e^2,\pm e^3,0\}$. Within Section 3, the notation $(h,\tau)$ always includes the scaling condition \eqref{scale}.  
\begin{Thm}\label{Thm-explicit}
The solution of \eqref{scale}-\eqref{explicit1} satisfies for all $1\le n\le T_\tau$,
\begin{eqnarray}\label{e-estimate1}
\norm g^n\norm_{\infty,h\Z^3}&\le& \norm g^0\norm_{\infty,h\Z^3}\le \norm f^0\norm_{L^\infty(\Omega)},
\\\label{e-estimate2}
\norm g^n\norm_{p,h\Z^3}^p&\le& \norm g^0\norm_{p, h\Z^3}^p+\sum_{m=0}^{n-1} \sum_{x\in h\Z^3}(D\cdot \tilde{u}^m(x)) |g^m(x)|^ph^3\tau \\\nonumber
&=& \norm g^0\norm_{p,\Omega_h}^p+\sum_{m=0}^{n-1} \sum_{x\in\Omega_h}(D\cdot \tilde{u}^m(x)) |g^m(x)|^ph^3\tau \\\nonumber
&\le&\norm f^0\norm_{L^p(\Omega)}^p+\sum_{m=0}^{n-1} \sum_{x\in\Omega_h}(D\cdot \tilde{u}^m(x))|g^m(x)|^p h^3\tau,\,\,\,\forall\,p\in[1,\infty). 
\end{eqnarray}
Furthermore, the comparison principle holds for the discrete problem: let $g^n$ be the solution of \eqref{scale}-\eqref{explicit1} and let  $\tilde{g}^n$ be that of \eqref{scale}-\eqref{explicit1} with initial data $\tilde{g}^0$; if $g^0\le\tilde{g}^0$ on $h\Z^3$, we have $g^n\le\tilde{g}^n$ on $h\Z^3$ for all $n=1,\ldots,T_\tau$.
\end{Thm}
\begin{proof}
Let $p^\ast$ be the H\"older conjugate of $p\in(1,\infty)$. We first observe that 
\begin{eqnarray*}
|g^0(x)|&=&\Big|\frac{1}{h^3}\int_{C_h(x)}\tilde{f}^0(y)dy\Big|\le \norm f^0\norm_{L^\infty(\Omega)},\quad \forall\, x\in h\Z^3,\\
|g^0(x)|^{p}&\le&\Big(\frac{1}{h^3}\int_{C_h(x)}|\tilde{f}^0(y)|dy\Big)^{p}
\le\Big\{\frac{1}{h^3}\Big(\int_{C_h(x)}|\tilde{f}^0(y)|^pdy\Big)^\frac{1}{p}\Big(\int_{C_h(x)}1^{p^\ast}dy\Big)^\frac{1}{p^\ast}\Big\}^{p}\\
&=&\frac{1}{h^3}\int_{C_h(x)}|\tilde{f}^0(y)|^pdy \quad (\mbox{the case of $p=1$ is also clear}),\\
\norm g^0\norm_{p,h\Z^3}&\le&\norm \tilde{f}^0\norm_{L^p(\R^3)}=\norm f^0\norm_{L^p(\Omega)}\le{\rm vol}(\Omega)^\frac{1}{p}\norm f^0\norm_{L^\infty(\Omega)},\,\,\,\,\forall\,p\in[1,\infty].
\end{eqnarray*}
For each $x\in h\Z^3$, \eqref{explicit1} is equivalent to 
\begin{eqnarray}\label{explicit1-1} 
g^{n+1}(x) &=&  \frac{1}{7}g^n(x)
+\sum_{j=1}^3\Big\{\Big(\frac{1}{7}+\frac{\tau}{2h}\tilde{u}^n_j(x)\Big)g^n(x-he^j)\\\nonumber
&& +\Big(\frac{1}{7}-\frac{\tau}{2h}\tilde{u}^n_j(x)\Big)g^n(x+he^j) \Big\}.
\end{eqnarray}
The truncation \eqref{truncation} and the scaling condition \eqref{scale} imply that 
\begin{eqnarray*}
\frac{1}{7}\pm\frac{\tau}{2h}\tilde{u}^n_j(x)\ge 0,\quad 
 \min_{x\in h\Z^3}g^n(x)\le g^{n+1}(x)\le \max_{x\in h\Z^3}g^n(x), 
\end{eqnarray*}
from which \eqref{e-estimate1} is derived by induction.   

We fix $x\in h\Z^3$ and re-write \eqref{explicit1-1} as 
$$g^{n+1}(x) = \sum_{\omega\in B} g^n(x+h \omega)\rho(\omega),\,\,\,\rho(\omega)\ge 0,\,\,\, \sum_{\omega\in B}\rho(\omega)=1.$$
It is clear that 
$$|g^{n+1}(x)| \le \sum_{\omega\in B} |g^n(x+h \omega)|\rho(\omega).$$
Applying the (discrete) H\"older inequality to the right hand side, we obtain 
\begin{eqnarray*}
&&|g^{n+1}(x)| \le \sum_{\omega\in B} |g^n(x+h \omega)|\rho(\omega)\le \Big(\sum_{\omega\in B}|g^n(x+h \omega)|^p\rho(\omega)\Big)^\frac{1}{p}\Big(\sum_{\omega\in B}1^{p^\ast}\rho(\omega)\Big)^\frac{1}{p^\ast},\\
&&|g^{n+1}(x)|^p\le \sum_{\omega\in B}|g^n(x+h \omega)|^p\rho(\omega),\quad\forall\,p\in(1,\infty),
\end{eqnarray*}
which leads to 
\begin{eqnarray}\label{explicit1-2}
&& |g^{n+1}(x)|^p \le  
 \frac{1}{7}\sum_{\omega\in B}|g^n(x+h \omega)|^p   -\sum_{j=1}^3\tilde{u}^n_j(x) D_j(|g^n|^p)(x)\tau
,\,\,\,\forall\,p\in[1,\infty),\\\nonumber
&& D_j(|g^n|^p)(x)=\frac{|g^n(x+he^j)|^p-|g^n(x-he^j)|^p}{2h}.
\end{eqnarray}
Nothing that $\tilde{u}^n_j$ is equivalently $0$ near $\p\Omega_h$ and $h\Z^3\setminus\Omega_h$, we sum up  \eqref{explicit1-2}$\times h^3$ over $x\in h\Z^3$ with the summation by parts (cf., \eqref{by-parts}) to obtain  
$$\norm g^{n+1}\norm_{p,h\Z^3}^p\le \norm g^n\norm_{p,h\Z^3}^p+\sum_{x\in h\Z^3}(D\cdot\tilde{u}^n(x))|g^n(x)|^ph^3\tau,$$
which yields \eqref{e-estimate2}.

Now, the comparison principle is obvious, since \eqref{explicit1} is linear with respect to the unknown $g^n$ and  $b^n:=\tilde{g}^n-g^n$ satisfies \eqref{explicit1} with $b$ in replace of $g$ and $b^0\ge0$. 
\end{proof}
We remark that the support of $g^n$ may not be uniformly  bounded for $(h,\tau)\to0$ under the generalized hyperbolic scaling.  
\subsection{Convergence of discrete problem (explicit)}

 Hereafter, $O(r)$ stands for a quantity whose absolute value is bounded by $M |r|$ as $r\to0$ with some constant $M\ge 0$ independent of $x$, $n$ and $r\to0$; $M_1,M_2,\ldots$ stand for some positive constants independent of $(h,\tau)$, $x$ and $n$.

We convert \eqref{explicit1} to a weak form. 
Fix an arbitrary test function $\varphi\in C^\infty([0,T]\times\R^3;\R)$ with supp$(\varphi)\subset[0,T)\times\R^3$ compact. Shifting $x$ to $x\mp h \omega$ in $\sum_{x\in h\Z^3}$ with the fact that $\varphi\equiv0$ outside a bounded domain of $\R^3$, we have  
\begin{eqnarray*}
&&\sum_{n=0}^{T_\tau-1}\sum_{x\in h\Z^3} \Big(g^{n+1}(x)-  \frac{1}{7}\sum_{\omega\in B}g^n(x+h \omega)  \Big)\frac{1}{\tau} \varphi(t_n,x) h^3\tau\\
&&\quad =\sum_{n=0}^{T_\tau-1}\sum_{x\in h\Z^3} \Big(g^{n+1}(x) \varphi(t_n,x) -  g^n(x)\frac{1}{7}\sum_{\omega\in B}\varphi(t_n,x-h \omega)  \Big)\frac{1}{\tau} h^3\tau\\
&&\quad=\sum_{n=0}^{T_\tau-1}\sum_{x\in h\Z^3} \Big(g^{n+1}(x) \varphi(t_n,x) -  g^n(x)\varphi(t_n,x) +g^n(x)O(h^2)  \Big)\frac{1}{\tau} h^3\tau\\
&&\quad=\sum_{n=0}^{T_\tau-1}\sum_{x\in h\Z^3} \Big(g^{n+1}(x) \varphi(t_{n+1},x) -  g^n(x)\varphi(t_n,x) \Big) h^3\\
&&\qquad -\sum_{n=0}^{T_\tau-1}\sum_{x\in h\Z^3} g^{n+1}(x)\frac{ \varphi(t_{n+1},x)- \varphi(t_{n},x)}{\tau} h^3\tau
 +\sum_{n=0}^{T_\tau-1}\sum_{x\in h\Z^3}g^n(x)O\Big(\frac{h^2}{\tau}\Big)  h^3\tau\\
 &&\quad =-\sum_{x\in h\Z^3} g^0(x)\varphi(0,x)h^3
  -\sum_{n=0}^{T_\tau-1}\sum_{x\in h\Z^3} g^{n+1}(x)\p_t\varphi(t_{n+1},x)h^3\tau+O(h^{\alpha}),
\end{eqnarray*}
where we also note that $\varphi\equiv0$ near $t=T$ and $g^n$ is bounded.   Similarly, we have 
\begin{eqnarray*}
&&\sum_{n=0}^{T_\tau-1}\sum_{x\in h\Z^3} \sum_{j=1}^3\tilde{u}^n_j(x) D_jg^n(x)\varphi(t_n,x) h^3\tau\\
&&\quad =\sum_{j=1}^3\sum_{n=0}^{T_\tau-1}\sum_{x\in h\Z^3} \tilde{u}^n_j(x) \frac{g^n(x+he^j)-g^n(x-he^j)}{2h}\varphi(t_n,x) h^3\tau\\
&&\quad =\sum_{j=1}^3\sum_{n=0}^{T_\tau-1}\sum_{x\in h\Z^3}
 \Big(\tilde{u}^n_j(x-he^j)g^n(x)\varphi(t_n,x-he^j)\\
&&\qquad  - \tilde{u}^n_j(x+he^j)g^n(x)\varphi(t_n,x+he^j)\Big)\frac{1}{2h} h^3\tau\\
&&\quad =- \sum_{j=1}^3\sum_{n=0}^{T_\tau-1}\sum_{x\in h\Z^3}
D_j\tilde{u}^n_j(x)g^n(x)\varphi(t_n,x-he^j)h^3\tau\\
&&\qquad - \sum_{j=1}^3\sum_{n=0}^{T_\tau-1}\sum_{x\in h\Z^3}
\tilde{u}^n_j(x+he^j)g^n(x)\p_{x_j}\varphi(t_n,x)h^3\tau\\
&&\qquad - \sum_{j=1}^3\sum_{n=0}^{T_\tau-1}\sum_{x\in h\Z^3}
\tilde{u}^n_j(x+he^j)g^n(x)O(h^2)h^3\tau.
\end{eqnarray*}
Therefore, we obtain the weak form of \eqref{explicit1} as
\begin{eqnarray}\label{weak}
0&=& \sum_{x\in h\Z^3} g^0(x)\varphi(0,x)h^3\\\nonumber
&&  +\sum_{n=0}^{T_\tau-1}\sum_{x\in h\Z^3} g^{n+1}(x)\p_t\varphi(t_{n+1},x)h^3\tau\\\nonumber
&&+ \sum_{j=1}^3\sum_{n=0}^{T_\tau-1}\sum_{x\in h\Z^3}
D_j\tilde{u}^n_j(x)g^n(x)\varphi(t_n,x-he^j)h^3\tau\\\nonumber
&& + \sum_{j=1}^3\sum_{n=0}^{T_\tau-1}\sum_{x\in h\Z^3}
\tilde{u}^n_j(x+he^j)g^n(x)\p_{x_j}\varphi(t_n,x)h^3\tau\\\nonumber
&& + \sum_{j=1}^3\sum_{n=0}^{T_\tau-1}\sum_{x\in h\Z^3}
\tilde{u}^n_j(x+he^j)g^n(x)O(h^2)h^3\tau \\\nonumber
&&+O(h^{\alpha}).
\end{eqnarray}
\indent We will show that \eqref{weak} yields \eqref{weak-sol} at the limit of $(h,\tau)\to0$. For this purpose, we investigate asymptotics of $u^n$ and $\tilde{u}^n$ for $(h,\tau)\to0$ under \eqref{scale}. It is well-known (see, e.g., \cite{Lady}) that any function $w\in H^1_0(\Omega;\R)$, $\Omega\subset\R^3$ satisfies 
$$w\in L^q(\Omega),\,\,\,\norm w\norm_{L^q(\Omega)}\le 48^{\frac{\gamma}{6}}\norm \p_xw\norm_{L^2(\Omega)}^\gamma\norm w\norm_{L^2(\Omega)}^{1-\gamma},\,\,\,\forall\,q\in[2,6],\,\,\,\gamma=\frac{3}{2}-\frac{3}{q}.$$
We, then, choose $q=3$ and apply it to our $v\in L^2([0,T];H^1_0(\Omega)^3)\cap L^\infty([0,T];L^2(\Omega)^3)$: by the H\"older inequality, 
\begin{eqnarray*}
\norm v_j(t,\cdot)\norm_{L^3(\Omega)}&\le& 48^{\frac{1}{2}}\norm \p_xv_j(t,\cdot)\norm_{L^2(\Omega)}^\frac{1}{2}\norm v_j(t,\cdot)\norm_{L^2(\Omega)}^{\frac{1}{2}},\\
\int_0^T\norm v_j(t)\norm_{L^3(\Omega)}^3dt
&\le&\int_0^T48^{\frac{3}{2}}\norm \p_xv_j(t,\cdot)\norm_{L^2(\Omega)}^\frac{3}{2}\norm v_j(t,\cdot)\norm_{L^2(\Omega)}^{\frac{3}{2}}dt\\
&\le&48^{\frac{3}{2}}\Big(\int_0^T \norm \p_x v_j(t,\cdot)\norm_{L^2(\Omega)}^2dt\Big)^\frac{3}{4}\Big(\int_0^T \norm  v_j(t,\cdot)\norm_{L^2(\Omega)}^6dt\Big)^\frac{1}{4} \\
&\le& 48^{\frac{3}{2}} T^\frac{1}{4}\norm \p_x v_j\norm_{L^2([0,T];L^2(\Omega))}^\frac{3}{2} \norm  v_j\norm_{L^\infty([0,T];L^2(\Omega))}^\frac{3}{2}\\
&\le&M_1\quad\mbox{ for $j=1,2,3$}.
\end{eqnarray*}
Due to  the estimates
\begin{eqnarray*}
&&|u^n_j(x)|^3\le\Big(\frac{1}{\tau h^3}\Big)^3
\Big(\int_{t^n}^{t_{n+1}} \int_{C_h(x)} |v_j(s,y)|dyds \Big)^3\\
&&\qquad \le\Big(\frac{1}{\tau h^3}\Big)^3
\Big\{
\Big(\int_{t^n}^{t_{n+1}} \int_{C_h(x)} 1^\frac{3}{2} dyds\Big)^\frac{2}{3}
\Big(\int_{t^n}^{t_{n+1}} \int_{C_h(x)} |v_j(s,y)|^3dyds\Big)^\frac{1}{3}
 \Big\}^3\\
 &&\qquad \le \frac{1}{\tau h^3}\int_{t^n}^{t_{n+1}} \int_{C_h(x)} |v_j(s,y)|^3dyds\quad\mbox{ for all $x\in\Omega_h$},\\
 &&\sum_{x\in\Omega_h}|u^n_j(x)|^3h^3\tau \le \int_0^T\norm v_j(t)\norm_{L^3(\Omega)}^3dt,
\end{eqnarray*}
we obtain 
\begin{eqnarray*}
\sum_{n=0}^{T_\tau-1}\sum_{x\in\Omega_h}|u^n_j(x)|^3h^3\tau\le M_1\quad\mbox{for all $(h,\tau)\to0$.}
\end{eqnarray*}
Now, we are ready to estimate the measure of the truncated part given in \eqref{truncation}, where we note that the truncation takes place only within $\Omega_h$:
\begin{eqnarray}\nonumber 
&& A^n_{\delta,j}:=\{x\in\Omega_h\,|\, |u^n_j(x)|>h^{-\beta}  \},\quad\delta:=(h,\tau),\quad
{\rm vol}(A^n_{\delta,j}):=\sum_{x\in A^n_{\delta,j}}h^3,\\\nonumber
&&3M_1\ge \sum_{j=1}^3\sum_{n=0}^{T_\tau-1}\sum_{x\in\Omega_h}|u^n_j(x)|^3h^3\tau 
\ge  \sum_{j=1}^3\sum_{n=0}^{T_\tau-1}\sum_{x\in A^n_{\delta,j}}h^{-3\beta}h^3\tau\\\nonumber
&&\qquad = h^{-3\beta} \sum_{j=1}^3\sum_{n=0}^{T_\tau-1}{\rm vol}(A^n_{\delta,j})\tau,\\\label{measure2}
&&\sum_{j=1}^3\sum_{n=0}^{T_\tau-1}{\rm vol}(A^n_{\delta,j})\tau
 \le 3M_1h^{3\beta}\quad\mbox{for all $(h,\tau)\to0$.}
\end{eqnarray}
Finally, we represent $D_ju_j^n(x)$ in terms of $\p_{x_j}\tilde{v}_j$. Let $\tilde{v}^\ep\in C^\infty((\ep,T-\ep)\times \R^3)$, $\ep>0$ be the $(t,x)$-mollification of $\tilde{v}$ with the standard mollifier on $\R^{4}$. Then, we have 
\begin{eqnarray*}
D_ju_j^n(x)&=&\frac{1}{\tau h^3}\frac{1}{2h}\int_{t_n}^{t_{n+1}}\int_{C_h(x)}\Big(\tilde{v}_j(s,y+he^j)-\tilde{v}_j(s,y-he^j)\Big)dyds\\
&=&\underline{\frac{1}{\tau h^3}\frac{1}{2h}\int_{t_n}^{t_{n+1}}\int_{C_h(x)}\Big\{\Big(\tilde{v}_j(s,y+he^j)-\tilde{v}_j(s,y-he^j)\Big)}\\
&&\qquad \underline{-\Big(\tilde{v}_j^\ep(s,y+he^j)-\tilde{v}_j^\ep(s,y-he^j)\Big)\Big\}dyds}_{\rm (i)}\\
&&+\underline{\frac{1}{\tau h^3}\frac{1}{2h}\int_{t_n}^{t_{n+1}}\int_{C_h(x)}\Big(\tilde{v}_j^\ep(s,y+he^j)-\tilde{v}_j^\ep(s,y-he^j)\Big)dyds,}_{\rm (ii)}\\
{\rm (i)}&\to&0\mbox{\quad as $\ep\to0+$ for each fixed $(h,\tau)$,}\\
{\rm (ii)}&=&\frac{1}{\tau h^3}\frac{1}{2h}\int_{t_n}^{t_{n+1}}\int_{C_h(x)}\int_{-h}^h\p_{x_j}\tilde{v}_j^\ep(s,y+ze^j)dzdyds\\
&=&\underline{\frac{1}{\tau h^3}\frac{1}{2h}\int_{-h}^h\int_{t_n}^{t_{n+1}}\int_{C_h(x)}\Big(\p_{x_j}\tilde{v}_j^\ep(s,y+ze^j)-\p_{x_j}\tilde{v}_j(s,y+ze^j)\Big)dydsdz}_{\rm (iii)}\\
&&+\frac{1}{\tau h^3}\frac{1}{2h}\int_{-h}^h\int_{t_n}^{t_{n+1}}\int_{C_h(x)}\Big(\p_{x_j}\tilde{v}_j(s,y+ze^j)-\p_{x_j}\tilde{v}_j(s,y)\Big)dydsdz\\
&&+\frac{1}{\tau h^3}\frac{1}{2h}\int_{-h}^h\int_{t_n}^{t_{n+1}}\int_{C_h(x)}\p_{x_j}\tilde{v}_j(s,y)dydsdz,\\
{\rm (iii)}&\to&0\mbox{\quad as $\ep\to0+$ for each fixed $(h,\tau)$}.
\end{eqnarray*}    
Therefore, we obtain 
\begin{eqnarray}\label{derivative}
&&D_ju_j^n(x)=\frac{1}{\tau h^3}\int_{t_n}^{t_{n+1}}\int_{C_h(x)}\p_{x_j}\tilde{v}_j(s,y)dyds\\\nonumber 
&&\qquad +\frac{1}{\tau h^3}\frac{1}{2h}\int_{-h}^h\int_{t_n}^{t_{n+1}}\int_{C_h(x)}\Big(\p_{x_j}\tilde{v}_j(s,y+ze^j)-\p_{x_j}\tilde{v}_j(s,y)\Big)dydsdz.
\end{eqnarray}

 We state convergence results. Define the step function $g_\delta:[0,\tau T_\tau+\tau)\times\R^3\to\R$ for each $\delta=(h,\tau)$ (note that $T_\tau\in\N$ is such that $T\in [\tau T_\tau,\tau T_\tau+\tau)$)  as  
\begin{eqnarray*}
g_\delta(s, y)&:=&
g^n(x),\,\,\, \mbox{ if }s\in[t_n,t_{n+1}),\,\,\, y\in C_h(x),\,\,\, x\in h\Z^3.
\end{eqnarray*}
\begin{Prop}[{\it weak convergence}]\label{Thm-weak}
There exists $\tilde{f}\in L^\infty([0,T],L^\infty(\R^3))$ to which  $\{g_\delta\}$ converges, up to subsequence, weakly$^\ast$ to $\tilde{f}$ in $L^\infty([0,T],L^\infty(\R^3))$ as $\delta=(h,\tau)\to0$ under the scaling condition \eqref{scale}. Furthermore, $\tilde{f}$ is the weak solution of \eqref{T2.eq} and $f:=\tilde{f}|_{x\in\Omega}$ is the weak solution of \eqref{T.eq} .  
\end{Prop}
\begin{proof}
Since $\{g_\delta\}$ is bounded in $L^\infty([0,T],L^\infty(\R^3))$ due to \eqref{e-estimate1}, we have a weakly$^\ast$ convergent subsequence, still denoted by $\{g_\delta\}$ with $\delta\to0$, and  its weak$^\ast$ limit  $\tilde{f}\in L^\infty([0,T],L^\infty(\R^3))$.    

We show that $\tilde{f}$ is the solution of \eqref{T2.eq} via \eqref{weak}.  Each term (except for the last one) of the right hand side of \eqref{weak} is denoted by (i)-(v). Observe that 
\begin{eqnarray*}
{\rm (i)}
&=&\sum_{x\in h\Z^3}\int_{C_h(x)} \tilde{f}^0(y)dy\times \varphi(0,x)\\
& =&\sum_{x\in h\Z^3}\int_{C_h(x)} \tilde{f}^0(y)\varphi(0,y)dy+\sum_{x\in  h\Z^3}\int_{C_h(x)} \tilde{f}^0(y)(\varphi(0,x)-\varphi(0,y))dy\\
& =&\sum_{x\in h\Z^3} \int_{C_h(x)}  \tilde{f}^0(y)\varphi(0,y)dy +O(h)
\to \int_{\R^3} \tilde{f}^0(y)\varphi(0,y)dy\mbox{\quad as $\delta\to0$.}
\end{eqnarray*}
Since $\p_t \varphi\in L^1([0,T];L^1(\R^3))$ with the support compact, we have 
\begin{eqnarray*}
{\rm (ii)}
&=& \sum_{n=0}^{T_\tau-1}\sum_{x\in h\Z^3}
\int_{t_n}^{t_{n+1}}\int_{C_h(x)}g_\delta(s,y)\p_t \varphi(s,y)dyds\\
&&+ \sum_{n=0}^{T_\tau-1}\sum_{x\in h\Z^3}
\int_{t_n}^{t_{n+1}}\int_{C_h(x)}g_\delta(s,y)(\p_t \varphi(t_{n+1},x)-\p_t \varphi(s,y))dyds\\
&=&\int_{0}^{T}\int_{\R^3}g_\delta(s,y)\p_t \varphi(s,y)dyds+O(|\delta|)\\
&\to&\int_{0}^{T}\int_{\R^3}\tilde{f}(s,y)\p_t \varphi(s,y)dyds\mbox{\quad as $\delta\to0$.}
\end{eqnarray*}
Since $|\tilde{u}^n_j(x)|\le |u^n_j(x)|$, we have 
\begin{eqnarray*}
|{\rm (v)}|
&\le & M_2h^2\norm g^n\norm_{\infty,h\Z^3} \sum_{j=1}^3\sum_{n=0}^{T_\tau-1}\sum_{x\in h\Z^3}\int_{t_n}^{t_{n+1}}\int_{C_h(x)}|\tilde{v}_j(s,y)|dyds\\
&\le& M_2h^2\norm f^0\norm_{L^\infty(\Omega)}\sum_{j=1}^3 \norm v_j\norm_{L^1([0,T];L^1(\Omega))}
\to0\mbox{\quad as $\delta\to0$.}
\end{eqnarray*}
\indent In order to deal with (iii), we first observe that 
\begin{eqnarray*}
&&D_j\tilde{u}^n_j(x)=\frac{1}{2}D^+_j\tilde{u}^n_j(x)+\frac{1}{2}D^-_j\tilde{u}^n_j(x),\\
&&\mbox{$x\in A^n_{\delta,j}$, $x+he^j\in A^n_{\delta,j}$ $\Rightarrow$ $|D^+_j\tilde{u}^n_j(x)|\le 2h^{-1-\beta}$},\\
&&\mbox{$x\in A^n_{\delta,j}$, $x+he^j\not\in A^n_{\delta,j}$ $\Rightarrow$ $|D^+_j\tilde{u}^n_j(x)|\le 2h^{-1-\beta}$},\\
&&\mbox{$x\not\in A^n_{\delta,j}$, $x+he^j\in A^n_{\delta,j}$ $\Rightarrow$ $|D^+_j\tilde{u}^n_j(x)|\le 2h^{-1-\beta}$},\\
&&\mbox{$x\not\in A^n_{\delta,j}$, $x+he^j\not\in A^n_{\delta,j}$ $\Rightarrow$ $D^+_j\tilde{u}^n_j(x)=D^+_ju^n_j(x)$},\\
&&\mbox{$x\in A^n_{\delta,j}$, $x-he^j\in A^n_{\delta,j}$ $\Rightarrow$ $|D^-_j\tilde{u}^n_j(x)|\le 2h^{-1-\beta}$},\\
&&\mbox{$x\in A^n_{\delta,j}$, $x-he^j\not\in A^n_{\delta,j}$ $\Rightarrow$ $|D^-_j\tilde{u}^n_j(x)|\le 2h^{-1-\beta}$},\\
&&\mbox{$x\not\in A^n_{\delta,j}$, $x-he^j\in A^n_{\delta,j}$ $\Rightarrow$ $|D^-_j\tilde{u}^n_j(x)|\le 2h^{-1-\beta}$},\\
&&\mbox{$x\not\in A^n_{\delta,j}$, $x-he^j\not\in A^n_{\delta,j}$ $\Rightarrow$ $D^-_j\tilde{u}^n_j(x)=D^-_ju^n_j(x)$}.
\end{eqnarray*}
Hence, defining $\tilde{A}^n_{\delta,j}:=\{x,x+he^j,x-he^j\,|\,x\in A^n_{\delta,j}\}$, we have 
\begin{eqnarray*}
&&\mbox{$x\in \Omega_h\setminus \tilde{A}^n_{\delta,j}$ $\Rightarrow$ $D_j\tilde{u}^n_j(x)=D_ju^n_j(x)$},\\
&&\mbox{$x\in \tilde{A}^n_{\delta,j}$ $\Rightarrow$ $|D_j\tilde{u}^n_j(x)|\le 2h^{-1-\beta}$},\\
&&{\rm vol}(\tilde{A}^n_{\delta,j})\le 3{\rm vol}(A^n_{\delta,j}).
\end{eqnarray*}
Therefore, noting that support of $u^n$ is contained in $\Omegain$,  we have 
\begin{eqnarray*}
{\rm (iii)}&=&\sum_{j=1}^3\sum_{n=0}^{T_\tau-1}\sum_{x\in \Omega_h}
D_j\tilde{u}^n_j(x)g^n(x)\varphi(t_n,x-he^j)h^3\tau\\
&=&\underline{ \sum_{j=1}^3 \sum_{n=0}^{T_\tau-1}\sum_{x\in\Omega_h\setminus \tilde{A}^n_{\delta,j}} D_ju^n_j(x)g^n(x)\varphi(t_n,x-he^j)h^3\tau}_{{\rm (iii)}_1}\\
&&+\underline{\sum_{j=1}^3 \sum_{n=0}^{T_\tau-1}\sum_{x\in \tilde{A}^n_{\delta,j}} D_j\tilde{u}^n_j(x)g^n(x)\varphi(t_n,x-he^j)h^3\tau.}_{{\rm (iii)}_2}
\end{eqnarray*}
By \eqref{measure2} and \eqref{scale}, we have 
\begin{eqnarray*}
|{{\rm (iii)}_2}|\le M_3 h^{-1+2\beta}\to0\mbox{\quad as $\delta\to0$}.
\end{eqnarray*}
Define 
$$\Omega^n_{\delta,j}:=\Omega_h\setminus \tilde{A}^n_{\delta,j},\,\,\,\Theta_{\delta,j}:=\bigcup_{n=0}^{T_r-1}\Big([t_n,t_{n+1})\times\bigcup_{x\in \Omega^n_{\delta,j}}C_h(x)\Big).$$
 It follows from \eqref{derivative} that 
\begin{eqnarray*}
{\rm(iii)}_1&=&\sum_{j=1}^3 \sum_{n=0}^{T_\tau-1}\sum_{x\in\Omega^n_{\delta,j}} \int_{t_n}^{t_{n+1}}\int_{C_h(x)}\p_{x_j}\tilde{v}_j(s,y)dyds\times g^n(x)\varphi(t_n,x-he^j)\\
&&\!\!\!\!\!\!\!\!\!\!\!\!\!\!\!\!\!\!+ \sum_{j=1}^3 \sum_{n=0}^{T_\tau-1}\sum_{x\in\Omega^n_{\delta,j}} \frac{1}{2h}\int_{-h}^h\int_{t_n}^{t_{n+1}}\int_{C_h(x)} \Big(\p_{x_j}\tilde{v}_j(s,y+ze^j)-\p_{x_j}\tilde{v}_j(s,y) \Big)dydsdz\\
&&\!\!\!\!\!\!\!\!\!\!\!\!\!\!\!\!\!\!\,\,\,\,\,\times g^n(x)\varphi(t_n,x-he^j)\\
&=&\underline{\sum_{j=1}^3 \iint_{\Theta_{\delta,j}}\p_{x_j}\tilde{v}_j(s,y) g_\delta(s,y)\varphi(s,y)dyds+O(h)}_{\rm(iii)_{11}} \\
&&\!\!\!\!\!\!\!\!\!\!\!\!\!\!\!\!\!\!+ \underline{\sum_{j=1}^3 \sum_{n=0}^{T_\tau-1}\sum_{x\in\Omega^n_{\delta,j}} \frac{1}{2h}\int_{-h}^h\int_{t_n}^{t_{n+1}}\int_{C_h(x)} \Big(\p_{x_j}\tilde{v}_j(s,y+ze^j)-\p_{x_j}\tilde{v}_j(s,y) \Big)dydsdz}\\
&&\!\!\!\!\!\!\!\!\!\!\!\!\!\!\!\!\!\!\,\,\,\,\,\underline{\times g^n(x)\varphi(t_n,x-he^j)}_{\rm(iii)_{12}}
\end{eqnarray*}
Since the measure of $\Theta_{\delta,j}$ tends to that of $[0,T]\times\Omega$ as $\delta\to0$, we have with the weak$^\ast$ convergence of $g_\delta$,  
\begin{eqnarray*}
{\rm(iii)_{11}} \to 
\sum_{j=1}^3 \iint_{[0,T]\times\Omega}\p_{x_j}\tilde{v}_j(s,y) \tilde{f}(s,y)\varphi(s,y)dyds=0\mbox{\quad as $\delta\to0$,}
\end{eqnarray*}
where $\nabla\cdot \tilde{v}=0$ by assumption. The latter term {\rm(iii)$_{12}$} is estimated as
\begin{eqnarray*}
|{\rm(iii)_{12}}|&\le& \sum_{j=1}^3M_4\frac{1}{2h}\int_{-h}^h\int_0^T\int_\Omega|\p_{x_j}\tilde{v}_j(s,y+ze^j)-\p_{x_j}\tilde{v}_j(s,y)|dydsdz\\
&\le& \sum_{j=1}^3M_4\sup_{|z|\le h}\norm \p_{x_j}\tilde{v}_j(\cdot,\cdot+ze^j)-\p_{x_j}\tilde{v}_j(\cdot,\cdot)\norm_{L^1([0,T];L^1(\Omega))}\\
&\to&0\mbox{\quad as $\delta\to0$.}
\end{eqnarray*}
Thus, we conclude that (iii)$\to0$ as $\delta\to0$.
 
Observe that 
\begin{eqnarray*}
{\rm (iv)}&=&\underline{\sum_{j=1}^3\sum_{n=0}^{T_\tau-1}\sum_{x\in\Omega_h}u^n_j(x+he^j)g^n(x)\p_{x_j}\varphi(t_n,x)h^3\tau}_{{\rm (iv)}_1}\\
&&+\underline{\sum_{j=1}^3\sum_{n=0}^{T_\tau-1}\sum_{x\in\Omega_h}\Big(\tilde{u}^n_j(x+he^j)-u^n_j(x+he^j)\Big)g^n(x)\p_{x_j}\varphi(t_n,x)h^3\tau,}_{{\rm (iv)}_2}\\
{\rm (iv)}_1&=&\sum_{j=1}^3\sum_{n=0}^{T_\tau-1}\sum_{x\in\Omega_h}\int_{t_n}^{t_{n+1}}\int_{C_h(x)}\tilde{v}_j(s,y+he^j) g_\delta(s,y)\p_{x_j}\varphi(s,y)dyds+O(h)\\
&=&\sum_{j=1}^3\int_0^T\int_\Omega \tilde{v}_j(s,y) g_\delta(s,y)\p_{x_j}\varphi(s,y)dyds\\
&&+\sum_{j=1}^3\int_0^T\int_\Omega \Big(\tilde{v}_j(s,y+he^j)-\tilde{v}_j(s,y)\Big) g_\delta(s,y)\p_{x_j}\varphi(s,y)dyds+O(h)\\
&\to&\sum_{j=1}^3\int_0^T\int_\Omega \tilde{v}_j(s,y) \tilde{f}(s,y)\p_{x_j}\varphi(s,y)dyds\mbox{\quad as $\delta\to0$},\\
|{\rm (iv)}_2|&=&\Big|\sum_{j=1}^3\sum_{n=0}^{T_\tau-1}\sum_{x\in\Omega_h}\Big(\tilde{u}^n_j(x)-u^n_j(x)\Big)g^n(x-he^j)\p_{x_j}\varphi(t_n,x-he^j)h^3\tau\Big|\\
&\le&M_5\sum_{j=1}^3\sum_{n=0}^{T_\tau-1}\sum_{x\in A^n_{\delta,j}} |u^n_j(x)|h^3\tau\\
&=&M_5\sum_{j=1}^3\sum_{n=0}^{T_\tau-1}\sum_{x\in A^n_{\delta,j}} \int_{t_n}^{t_{n+1}}\int_{C_h(x)}|v_j(s,y)|dyds\\
&\to&0\mbox{\quad as $\delta\to0$},
\end{eqnarray*}
where we used \eqref{measure2} for the last convergence. 
We conclude that $\tilde{f}$ satisfies \eqref{weak-sol2} to be  the weak solution of \eqref{T2.eq}. As stated in Introduction, $\tilde{f}|_{\R^3\setminus\Omega}=0$ and $f=\tilde{f}|_{x\in\Omega}$ satisfies \eqref{weak-sol} to be the weak solution of  \eqref{T.eq}. 
\end{proof}
\begin{Prop}\label{3Lp}
The sequence $\{g_\delta\}$ is bounded in $L^p([0,T];L^p(\R^3))$ for all $1\le p< \infty$ with 
$$\limsup_{\delta\to0}\norm g_\delta\norm_{L^p([0,T];L^p(\R^3))}\le T^\frac{1}{p}\norm \tilde{f}^0\norm_{L^p(\R^3)}=T^\frac{1}{p}\norm f^0\norm_{L^p(\Omega)},$$
where  $\delta=(h,\tau)\to0$ is taken under  the scaling condition \eqref{scale}.   
\end{Prop}
\begin{proof}
Define $\Omega^n_\delta:=\cup_{j=1,2,3}\Omega^n_{\delta,j}$ and $A^n_\delta:=\cup_{j=1,2,3}\tilde{A}^n_{\delta,j}$. 
 It follows from \eqref{e-estimate2} that for any $1\le p<\infty$, 
\begin{eqnarray*}
\norm g_\delta\norm_{L^p([0,T];L^p(\R^3))}^p&=&\sum_{n=0}^{T_\tau}\sum_{x\in h\Z^3}|g^n(x)|^ph^3\tau=\sum_{n=0}^{T_\tau}\norm g^n\norm_{p,h\Z^3}^p\tau\\
&\le& (T+\tau)\norm \tilde{f}^0\norm_{L^p(\R^3)}^p+\sum_{n=1}^{T_\tau}\sum_{m=0}^{n-1}\sum_{x\in\Omega_h} D\cdot\tilde{u}^m(x)|g^m(x)|^ph^3\tau^2\\
&=&(T+\tau)\norm f^0\norm_{L^p(\Omega)}^p
+\underline{\sum_{n=1}^{T_\tau}\sum_{m=0}^{n-1}\sum_{x\in\Omega^n_\delta} D\cdot u^m(x)|g^m(x)|^ph^3\tau^2}_{\rm (i)}\\
&&+\underline{\sum_{n=1}^{T_\tau}\sum_{m=0}^{n-1}\sum_{x\in A^n_\delta} D\cdot\tilde{u}^m(x)|g^m(x)|^ph^3\tau^2}_{\rm (ii)}.
\end{eqnarray*} 
 Due to \eqref{e-estimate1} and \eqref{measure2}, we have 
\begin{eqnarray*}
|{\rm (ii)}|&\le&M_6h^{-1+2\beta}\to0\mbox{\quad as $\delta\to0$}.
\end{eqnarray*}
Due to \eqref{derivative} and $\nabla\cdot \tilde{v}=0$, we have 
\begin{eqnarray*}
{\rm (i)}&=&\sum_{n=1}^{T_\tau}\sum_{m=0}^{n-1}\sum_{x\in\Omega^n_\delta} \int_{t_m}^{t_{m+1}}\int_{C_h(x)}\nabla\cdot \tilde{v}(s,y)dyds\times|g^m(x)|^p\tau\\
&&+\sum_{n=1}^{T_\tau}\sum_{m=0}^{n-1}\sum_{x\in\Omega^n_\delta} \sum_{j=1}^3\frac{1}{2h}\int_{-h}^h\int_{t_m}^{t_{m+1}}\int_{C_h(x)}   \Big(\p_{x_j}\tilde{v}(s,y+ze^j)-\p_{x_j}\tilde{v}(s,y) \Big)dydsdz\\
&&\quad \times |g^m(x)|\tau,\\
|{\rm (i)}|&\le& M_7\sum_{j=1}^3\sup_{|z|\le h}\norm \p_{x_j}\tilde{v}(\cdot,\cdot+ze^j)-\p_{x_j}\tilde{v}(\cdot,\cdot)\norm_{L^1([0,T];L^1(\Omega))}\\
&\to&0\mbox{\quad as $\delta\to0$}.
\end{eqnarray*}
\end{proof}
\begin{Thm}[{\it strong convergence}]\label{Thm-strong}
Let $f$, $\tilde{f}$ be the weak solution of \eqref{T.eq}, \eqref{T2.eq}, respectively. Then, 
$\{g_\delta\}$ converges to $\tilde{f}$ strongly in $L^2([0,T];L^2(\R^3))$ as $\delta=(h,\tau)\to0$ (in the whole sequence with the scaling condition \eqref{scale});  
 $\{g_\delta|_{x\in\Omega}\}$ converges to $f$ strongly in $L^p([0,T];L^p(\Omega))$ for any $1\le p<\infty$ as $\delta=(h,\tau)\to0$ (in the whole sequence under the scaling condition \eqref{scale}). 
\end{Thm}
 \begin{proof}
 Due to Proposition \ref{3Lp}, we see that $\{g_\delta\}$ is bounded in $L^2([0,T];L^2(\R^3))$ and the proof of Proposition \ref{Thm-weak} is valid with the weak convergence of  $\{g_\delta\}$ in $L^2([0,T];L^2(\R^3))$, where the weak limit $\tilde{f}\in L^2([0,T];L^2(\R^3))$ satisfies 
 \begin{eqnarray}\label{3strong}
 \norm \tilde{f}\norm_{L^2([0,T];L^2(\R^3))}\le \liminf_{\delta\to0}  \norm g_\delta\norm_{L^2([0,T];L^2(\Omega))}.
 \end{eqnarray} 
As stated in (F1), the exact weak solution $\tilde{f}\in L^\infty([0,T];L^\infty(\R^3))$ of \eqref{T2.eq},  being regarded as a weak solution of $L^2([0,T];L^2(\R^3))$, satisfies 
 \begin{eqnarray*}
 \norm \tilde{f}(t,\cdot)\norm_{L^2(\R^3)}\equiv\norm \tilde{f}^0\norm_{L^2(\R^3)}=\norm f^0\norm_{L^2(\Omega)},\quad
\norm \tilde{f}\norm_{L^2([0,T];L^2(\R^3))}=T^\frac{1}{2}\norm f^0\norm_{L^2(\Omega)}.
 \end{eqnarray*}
Hence, by Proposition \ref{3Lp} with $p=2$, we have 
\begin{eqnarray}\label{3strong-2}
 \limsup_{\delta\to0}\norm g_\delta\norm_{L^2([0,T];L^2(\R^3))}\le T^\frac{1}{2}\norm f^0\norm_{L^2(\Omega)}=\norm \tilde{f}\norm_{L^2([0,T];L^2(\R^3))}.
 \end{eqnarray}
The $L^2$-strong convergence of $\{g_\delta\}$ to $\tilde{f}$ follows from \eqref{3strong} and \eqref{3strong-2}, where the uniqueness of \eqref{T2.eq} implies that the whole sequence $\{g_\delta\}$ must converge to $\tilde{f}$. 

Let $\{g_\delta\}$ be a subsequence that converges to $\tilde{f}$  pointwise a.e. on $[0,T]\times\R^3$.  Then, $|g_\delta(t,x)-\tilde{f}(t,x)|^p=|g_\delta(t,x)-f(t,x)|^p$  on $[0,T]\times\Omega$, which  converges to $0$ a.e. $(t,x)\in[0,T]\times\Omega$ for any $1\le p<\infty$. 
Since $|g_\delta|$ and $|\tilde{f}|$ are bounded by $\norm f^0\norm_{L^\infty(\Omega)}$, Lebesgue's dominated convergence theorem yields
  $$\norm g_\delta-f\norm_{L^p([0,T];L^p(\Omega))}^p=\int_0^T\int_\Omega|g_\delta(t,x)-f(t,x)|^pdxdt\to0\quad \mbox{ \,\,\, as $\delta\to0$}. $$
If the whole sequence $\{g_\delta|_{x\in\Omega}\}$ does not converges to $f$ in $L^p([0,T];L^p(\Omega))$, we have a subsequence $\{g_{\delta'}\}$ such that $\norm g_{\delta'}-f\norm_{L^p([0,T];L^p(\Omega))}\ge\exists\, \ep_0>0$ for all $\delta'$. However, repeating the above argument for $\{g_{\delta'}\}$, we find a subsequence of $\{g_{\delta'}|_{x\in\Omega}\}$ that converges to $f$ in $L^p([0,T];L^p(\Omega))$, which is a contradiction.  
 \end{proof}
\setcounter{section}{3}
\setcounter{equation}{0}
\section{Implicit method}

In this section, we investigate an implicit finite difference  method of \eqref{T.eq}, where the discrete problem is formulated on $\Omega_h$ with the artificial $0$-boundary condition.  {\it Our implicit scheme is ``non-local'' in the sense that the presence of $h\Z^3\setminus\Omega_h$ does change point-wise  features.} 
Consider initial data $f^0\in L^2(\Omega)$ such that $f^0\equiv 0$ in a neighborhood of $\p\Omega$. Define $g^0:\Omega_h\to\R$ as 
$$g^0(x):=\frac{1}{ h^3}\int_{C_h(x)}f^0(y)\,dy.$$ 
Note that $\norm g^0\norm_{2,\Omega_h}\le \norm f^0\norm_{L^2(\Omega)}$. 
Fix $v\in L^2([0,T];H^1_0(\Omega)^3)\cap L^\infty([0,T];L^2(\Omega)^3)$  such that $\nabla\cdot v=0$, and  supp$(v)\subset[0,T]\times \Omega$. Define $u^n=(u^n_1,u^n_2,u^n_3):\Omega_h\to\R^3$, $n=0,1,\ldots, T_\tau-1$  as 
\begin{eqnarray}\label{impli222}
u^n_i(x):=\frac{1}{\tau h^3}\int_{\tau n}^{\tau(n+1)}\int_{C_h(x)}v_i(s,y)\,dyds. 
\end{eqnarray}
Note that $g^0\equiv 0$, $u^n\equiv0$ near the boundary of $\Omega_h$  for all sufficiently small $h>0$; we always consider such $h>0$; we extend $u^n$ to be $0$ outside $\Omega_h$, where \eqref{impli222} still holds with the $0$-extension of $v$ outside $\Omega$.      

\subsection{Discrete Helmholtz-Hodge decomposition}

Our discussion is based on the (discrete) $L^2$-estimate, where we verify a discrete version of the formal calculation  
$$ \int_\Omega(\p_t f +v\cdot\nabla f)fdx=\frac{d}{dt}\norm f(t,\cdot)\norm_{L^2(\Omega)}^2=0.$$
Here, $\nabla\cdot v=0$ plays an essential role. Unfortunately, the above $u^n=(u^n_1,u^n_2,u^n_3)$ does not satisfy $D\cdot u^n=0$ (nor $D^\pm\cdot u^n=0$) in general, even though $D\cdot u^n$ is asymptotically vanishing as seen in Section 3. In order to have a discrete divergence-free constraint, we operate the discrete Helmholtz-Hodge decomposition to $u^n$ and design a discrete problem with the divergence-free part of $u^n$. We will see that the potential part of $u^n$  is asymptotically vanishing.     

We state the discrete Helmholtz-Hodge decomposition obtained in \cite{Kuroki-Soga}, where the idea is originally introduced by Chorin \cite{Chorin}.     
\begin{Lemma}[\cite{Kuroki-Soga}]\label{Projection}
For each function  $u:\Omega_h\to\R^3$, there exist unique functions $w:\Omega_h\to\R^3$ and  $\phi:\Omega_h\to\R$ such that 
\begin{eqnarray}\label{HHD}
&&D^-\cdot w=0,\quad 
w+D^+ \phi =u \mbox{ on $\Omega_h\setminus \partial\Omega_h$};\quad w=0,\quad  \phi=0 \mbox{ on $\partial\Omega_h$},\\\nonumber
&&\norm w\norm_{2,\Omega_h}^2\le \sum_{x\in\Omegain}|u(x)|^2h^3,\quad
\norm D^+ \phi \norm_{2,\Omega_h}^2\le \sum_{x\in\Omegain}|u(x)|^2h^3,
\end{eqnarray}
where $u$ does not necessarily need to vanish on $\partial \Omega_h$.  If $u|_{\partial\Omega_h}=0$, we have 
\begin{eqnarray}\label{242424242}
\sum_{x\in \Omega_h\setminus \partial\Omega_h}|u(x)-w(x)|^2\le A\sum_{x\in\Omega_h\setminus \partial\Omega_h}|D^-\cdot u(x)|^2, 
\end{eqnarray}
where $A>0$ is a constant depending only on $\Omega$.
\end{Lemma}
\noindent The function $w$ in \eqref{HHD} is denoted by $P_h u$ and $P_h$ is called the discrete  Helmholtz-Hodge projection.  We remark that $P_h u$ is numerically detective (see the proof  in \cite{Kuroki-Soga}); one can formulate  the discrete Helmholtz-Hodge decomposition also with the central difference \cite{Maeda-Soga}.   
\subsection{Discrete problem (implicit)}
We define $w^n:=P_h u^n$. Our discrete problem  with unknowns $g^1,\ldots,g^{T_\tau}:\Omega_h\to\R$ is: 
\begin{eqnarray}\label{implicit1}
&&\frac{g^{n+1}(x)-g^n(x)}{\tau}
 +\frac{1}{2}\sum_{j=1}^3\Big(w_j^n(x-he^j)D^+_jg^{n+1}(x-he^j)\\\nonumber
&&\qquad\qquad\qquad\qquad\qquad+w_j^n(x)D^+_jg^{n+1}(x)\Big)=0,\quad x\in\Omegain,\\\label{implicit2}
&&g^{n+1}|_{\p\Omega_h}=0.
\end{eqnarray}
We remark that the whole argument on \eqref{implicit1}-\eqref{implicit2} is free from any scaling condition of $h, \tau$; the price to pay is that the comparison principle is not clear.  In the upcoming $L^2$-estimate, our complicated choice of the advection term in \eqref{implicit1} provides null-contribution, where we note that this idea is originally given by Ladyzhenskaya \cite{Lady} in her discrete  approximation of the incompressible Navier-Stokes equations.  
\begin{Thm}\label{implicit}
The equation \eqref{implicit1}-\eqref{implicit2} is uniquely solvable for any $h,\tau$. Furthermore, the solution satisfies 
\begin{eqnarray}\label{L2-estimate}
\norm g^n\norm_{2,\Omega_h}\le \norm g^0\norm_{2,\Omega_h}\le \norm f^0\norm_{L^2(\Omega)},\quad \forall\,n=0,1,\ldots,T_\tau.
\end{eqnarray}
\end{Thm}
\begin{proof}
To verify the first assertion, it is enough to prove that \eqref{implicit1}-\eqref{implicit2} is uniquely solvable with respect to $g^{n+1}$, provided $g^n:\Omega_h\to\R$ with   $g^n|_{\p\Omega_h}=0$ is given. 
We label the elements of $\Omega_h\setminus\partial\Omega_h$ as $x^{1},x^{2},\ldots,x^{a}$. Introduce $y,b\in\R^{a}$ as  
\begin{eqnarray*}
y:=\big(g^{n+1}(x^{1}),\ldots,g^{n+1}(x^{a})   \big),\quad b:=\big(g^{n}(x^{1}),\ldots,g^{n}(x^{a})\big).
\end{eqnarray*}
Then,  \eqref{implicit1}-\eqref{implicit2} is equivalent to the linear equations 
$$A(g^n;h,\tau)y=b,$$
where $A(g^n;h,\tau)$  is a $(a\times a)$-matrix depending on $g^n,h,\tau$. 
We prove that the matrix $A(g^n;h,\tau)$ is invertible, i.e.,    $A(g^n;h,\tau)y=0$ has the unique solution $y=0$.  Let $y=y_0$ be a solution of $A(g^n;h,\tau)y=0$. Then, we have $g^{n+1}:\Omega_h\to\R$ with $g^{n+1}|_{\partial\Omega_h}=0$ such that    
\begin{eqnarray*}
g^{n+1}(x)=-\tau\sum_{j=1}^3 \frac{w^n_j(x-he^j)D_j^+g^{n+1}(x-he^j)+w^n_j(x)D_j^+g^{n+1}(x)}{2},\quad x\in\Omega_h\setminus\partial\Omega_h.
\end{eqnarray*}
We take the inner product: noting  the $0$-boundary condition of $g^{n+1}$, we have 
\begin{eqnarray*}
&&(g^{n+1},g^{n+1})_{2,\Omega_h}=\norm g^{n+1}\norm_{2,\Omega_h}^2\\
&&\quad =-\frac{\tau}{2}\sum_{j=1}^3 \sum_{x\in\Omega_h\setminus\partial\Omega_h}\Big(w^n_j(x-he^j)D_j^+g^{n+1}(x-he^j)
 + w^n_j(x)D_j^+g^{n+1}(x)\Big)g^{n+1}(x)h^3\\
&&\quad  =-\frac{\tau}{2}\sum_{j=1}^3 \sum_{x\in\Omega_h\setminus\partial\Omega_h}\Big(w^n_j(x-he^j)\frac{g^{n+1}(x)-g^{n+1}(x-he^j)}{h}\\
&&\qquad\qquad\qquad\qquad +w^n_j(x)\frac{g^{n+1}(x+he^j)-g^{n+1}(x)}{h} \Big)g^{n+1}(x)h^3\\
&&\quad =-\frac{\tau}{2}\sum_{j=1}^3 \sum_{x\in\Omega_h\setminus\partial\Omega_h}
-\frac{w^n_j(x)-w^n_j(x-he^j)}{h}g^{n+1}(x)^2h^3\\
&&\qquad -\frac{\tau}{2}\sum_{j=1}^3 \sum_{x\in\Omega_h}
\frac{1}{h}w^n_j(x)g^{n+1}(x+he^j)g^{n+1}(x)h^3\\
&&\qquad +\frac{\tau}{2}\sum_{j=1}^3 \sum_{x\in\Omega_h} \frac{1}{h}w^n_j(x-he^j)g^{n+1}(x-he^j)g^{n+1}(x)h^3,
\end{eqnarray*} 
where $w^n$ is extended to be $0$ outside $\Omega_h$. Shifting $x$ to $x+he^j$ in the last summation with the $0$-boundary condition, we  obtain 
 \begin{eqnarray*}
\norm g^{n+1}\norm_{2,\Omega_h}^2=\frac{\tau}{2}\sum_{x\in\Omega_h\setminus\partial\Omega_h}\big(D^-\cdot w^n(x)\big)|u^{n+1}(x)|^2h^3=0,\quad {\rm i.e.,}\,\,\,\,y_0=0.
\end{eqnarray*} 
Therefore,  \eqref{implicit1}-\eqref{implicit2} is uniquely solvable with 
$$ \norm g^{n+1}\norm_{2,\Omega_h}^2\le (g^{n+1},g^{n})_{2,\Omega_h}\le  \norm g^{n+1}\norm_{2,\Omega_h} \norm g^{n}\norm_{2,\Omega_h},$$ 
which yields $\norm g^{n+1}\norm_{2,\Omega_h} \le \norm g^{n}\norm_{2,\Omega_h}\le \norm g^{n-1}\norm_{2,\Omega_h}\le\cdots\le\norm g^0\norm_{2,\Omega_h}\le \norm f^0\norm_{L^2(\Omega)}$.
\end{proof}
\subsection{Convergence of discrete problem (implicit)} 

We derive the weak form of \eqref{implicit1}-\eqref{implicit2}. 
Fix an arbitrary test function $\varphi\in C^\infty([0,T]\times\R^3)$ with supp$(\varphi)\subset [0,T)\times\R^3$ compact. Set $\varphi^n:=\varphi(\tau n,\cdot)$, where  $\varphi^{T_\tau}\equiv0$. We extend $w^n$ and $g^n$ to be $0$ outside $\Omega_h$.  Observe that 
\begin{eqnarray*}
&&\!\!\!\!\!\!\!\!\!\!(g^{n+1},\varphi^{n+1})_{2,\Omega_h}=(g^n,\varphi^{n})_{2,\Omega_h}+\tau\Big(g^n,\frac{\varphi^{n+1}-\varphi^n}{\tau}\Big)_{2,\Omega_h}\\
&&-\frac{\tau}{2}\sum_{j=1}^3\sum_{x\in\Omegain} \Big(w^n_j(x-he^j)D_j^+g^{n+1}(x-he^j)+ w^n_j(x)D_j^+g^{n+1}(x)\Big)\varphi^{n+1}(x)h^3.
\end{eqnarray*}
Due to the discrete divergence-free constraint of $w^n$ and the $0$-boundary condition of $g^n$ and $w^n$, we have 
\begin{eqnarray*}
&&\sum_{j=1}^3\sum_{x\in\Omegain} \Big(w^n_j(x-he^j)D_j^+g^{n+1}(x-he^j)+ w^n_j(x)D_j^+g^{n+1}(x)\Big)\varphi^{n+1}(x)h^3\\
&&=\sum_{j=1}^3\sum_{x\in\Omegain} 
\Big(w^n_j(x-he^j)\frac{g^{n+1}(x)-g^{n+1}(x-he^j)}{h}\\
&&\qquad+w^n_j(x)\frac{g^{n+1}(x+he^j)-g^{n+1}(x)}{h} \Big)\varphi^{n+1}(x)h^3\\
&&=\sum_{j=1}^3 \sum_{x\in\Omega_h\setminus\partial\Omega_h}
-\frac{w^n_j(x)-w^n_j(x-he^j)}{h}g^{n+1}(x)\varphi^{n+1}(x)h^3\\
&&\qquad+\sum_{j=1}^3 \sum_{x\in\Omega_h}
\frac{1}{h}w^n_j(x)g^{n+1}(x+he^j)\varphi^{n+1}(x)h^3\\
&&\qquad-\sum_{j=1}^3 \sum_{x\in\Omega_h} \frac{1}{h}w^n_j(x-he^j)g^{n+1}(x-he^j)\varphi^{n+1}(x)h^3\\
&&=\sum_{j=1}^3 \sum_{x\in\Omega_h}
\frac{1}{h}w^n_j(x)g^{n+1}(x+he^j)\varphi^{n+1}(x)h^3\\
&&\qquad-\sum_{j=1}^3 \sum_{x\in\Omega_h} \frac{1}{h}w^n_j(x)g^{n+1}(x)\varphi^{n+1}(x+he^j)h^3\\
&&=-\sum_{j=1}^3 \sum_{x\in\Omega_h} \frac{1}{h}\Big(w^n_j(x)g^{n+1}(x)\varphi^{n+1}(x+he^j)  -     w^n_j(x)g^{n+1}(x)\varphi^{n+1}(x) \Big)h^3\\
&&\qquad+\sum_{j=1}^3 \sum_{x\in\Omega_h}
\frac{1}{h}\Big(w^n_j(x)g^{n+1}(x+he^j)\varphi^{n+1}(x)  -w^n_j(x)g^{n+1}(x)\varphi^{n+1}(x) \Big)h^3\\
&&=-\sum_{j=1}^3 \sum_{x\in\Omega_h} w^n_j(x)g^{n+1}(x)D_j^+\varphi^{n+1}(x)h^3\\
&&\qquad+\sum_{j=1}^3 \sum_{x\in\Omega_h}
\frac{1}{h}\Big(w^n_j(x-he^j)g^{n+1}(x)\varphi^{n+1}(x-he^j)  -w^n_j(x)g^{n+1}(x)\varphi^{n+1}(x) \Big)h^3\\
&&=-\sum_{j=1}^3 \sum_{x\in\Omega_h} w^n_j(x)g^{n+1}(x)D_j^+\varphi^{n+1}(x)h^3\\
&&\qquad+\sum_{j=1}^3 \sum_{x\in\Omega_h}
\frac{1}{h}w^n_j(x-he^j)g^{n+1}(x)\Big(\varphi^{n+1}(x-he^j) -\varphi^{n+1}(x)\Big)h^3\\
&&\qquad-\sum_{j=1}^3 \sum_{x\in\Omega_h} \frac{ w^n_j(x)-w^n_j(x-he^j)}{h}g^{n+1}(x)\varphi^{n+1}(x) h^3\\
&&=-\sum_{j=1}^3 \sum_{x\in\Omega_h} \Big(w^n_j(x)g^{n+1}(x)D_j^+\varphi^{n+1}(x)  + w^n_j(x)u^{n+1}(x+he^j)D_j^+\varphi^{n+1}(x)  \Big)h^3\\
&&=-\sum_{j=1}^3\Big\{\Big(w^n_j(\cdot)g^{n+1}(\cdot),D_j^+\varphi^{n+1}(\cdot) \Big)_{2,\Omega_h} + \Big(w^n_j(\cdot)u^{n+1}(\cdot+he^j),D_j^+\varphi^{n+1}(\cdot)  \Big)_{2,\Omega_h}\Big\}.
\end{eqnarray*} 
Hence, taking summation with respect to $n$ and noting that $\varphi^{T_\tau}=0$, we have 
\begin{eqnarray}\label{weak-implicit}
&&(g^0,\varphi^0)_{2,\Omega_h}+\sum_{n=0}^{T_\tau-1}  \Big(g^n,\frac{\varphi^{n+1}-\varphi^n}{\tau}\Big)_{2,\Omega_h}\tau  
+\sum_{j=1}^3 \sum_{n=0}^{T_\tau-1}\Big\{\Big(w^n_j(\cdot)g^{n+1}(\cdot),D_j^+\varphi^{n+1}(\cdot) \Big)_{2,\Omega_h}  \\\nonumber
&&+ \Big(w^n_j(\cdot)g^{n+1}(\cdot+he^j),D_j^+\varphi^{n+1}(\cdot)  \Big)_{2,\Omega_h}\Big\}\frac{\tau}{2}=0.
\end{eqnarray} 
We show that \eqref{weak-implicit} yields \eqref{weak-sol} at the limit of $(h,\tau)\to0$. 

Let $g_\delta:[0,\tau T_\tau+\tau)\times\Omega\to\R$, $\delta=(h,\tau)$ be the step function generated by the solution of \eqref{implicit1}-\eqref{implicit2} as 
\begin{eqnarray*}
g_\delta(s, y):=\begin{cases}
g^n(x),\,\,\, \mbox{ if }s\in[t_n,t_{n+1}),\,\,\, y\in C_h(x),\,\,\, x\in \Omega_h,\\
0,\,\,\,\,\mbox{ otherwise}.
\end{cases}
\end{eqnarray*}
\begin{Thm}[{\it strong convergence}]\label{implicit-convergence}
The sequence $\{g\}_\delta$ converge to the solution of \eqref{T.eq} strongly in $L^2([0,T];L^2(\Omega))$ as $\delta=(h,\tau)\to0$ (in the whole sequence without any scaling condition). 
\end{Thm}
\begin{proof}
By \eqref{L2-estimate}, the sequence $\{g_\delta\}$ is bounded in  $L^2([0,T];L^2(\Omega))$. There exist a subsequence, still denoted by $\{g_\delta\}$, and a function $f\in L^2([0,T];L^2(\Omega))$ such that  $\{g_\delta\}$ converges to $f$ weakly as $\delta\to0$. 

We first show that  $f$ is the solution of \eqref{T.eq}. The  terms of the left hand side of  \eqref{weak-implicit} are denoted by (a), (b), (c), respectively. Hereafter, $O(r)$ stands for a quantity whose absolute value is bounded by $M |r|$ with some constant $M\ge0$ independent of $x,n$ and $r\to0$; $M_1,M_2,\ldots$ stand for some positive constants independent of $\delta=(h,\tau)$, $x$ and $n$.

It is clear that  
$${\rm (a)}\to\int_{\Omega}f^0(y)\varphi(0,y)dy,\quad {\rm (b)}\to\int_0^T\int_\Omega f(s,y)\p_t\varphi(s,y)dyds\mbox{\quad as $\delta\to0$}.$$
Observe  that  
\begin{eqnarray*}
&&\Big(w^n_j(\cdot)g^{n+1}(\cdot+he^j),D_j^+\varphi^{n+1}(\cdot)  \Big)_{2,\Omega_h}
=\Big(u^n_j(\cdot)g^{n+1}(\cdot+he^j),D_j^+\varphi^{n+1}(\cdot)  \Big)_{2,\Omega_h}\\
&&\qquad +\underline{\Big((w^n_j(\cdot)-u^n_j(\cdot))g^{n+1}(\cdot+he^j),D_j^+\varphi^{n+1}(\cdot)  \Big)_{2,\Omega_h}}_{\rm(i)}\\
&&\quad = \Big(u^n_j(\cdot+he^j)g^{n+1}(\cdot+he^j),D_j^+\varphi^{n+1}(\cdot+he^j)  \Big)_{2,\Omega_h}\\
&&\qquad + \underline{\Big((u^n_j(\cdot)-u^n_j(\cdot+he^j))g^{n+1}(\cdot+he^j),D_j^+\varphi^{n+1}(\cdot+he^j)  \Big)_{2,\Omega_h}}_{\rm(ii)}\\
&&\qquad+\underline{\Big(u^n_j(\cdot)g^{n+1}(\cdot+he^j),D_j^+\varphi^{n+1}(\cdot)-D_j^+\varphi^{n+1}(\cdot+he^j)  \Big)_{2,\Omega_h}}_{\rm(iii)}+{\rm(i)}\\
&&\quad = \Big(u^n_j(\cdot)g^{n+1}(\cdot),D_j^+\varphi^{n+1}(\cdot)  \Big)_{2,\Omega_h}+{\rm(i)}+{\rm(ii)}+{\rm(iii)}.
\end{eqnarray*}
By \eqref{242424242}, we have 
\begin{eqnarray*}
&&|{\rm(i)}|\le M_1\norm g^{n+1}\norm_{2,\Omega_h}\norm \norm D^-\cdot u^n\norm_{2,\Omega_h}\le  M_1\norm f^0\norm_{L^p(\Omega)} \norm D^-\cdot u^n\norm_{2,\Omega_h}.
\end{eqnarray*}
Reasoning similar to the proof of  \eqref{derivative} yields 
\begin{eqnarray*}
&&D^-\cdot u^n(x)=\frac{1}{\tau h^3}\int_{t_n}^{t_{n+1}}\int_{C_h(x)}\nabla\cdot v(s,y)dyds\\
&&\qquad+\sum_{j=1}^3\frac{1}{\tau h^3}\frac{1}{h}\int^{0}_{-h}\int_{t_n}^{t_{n+1}}\int_{C_h(x)} \Big(\p_{x_j}v_j(s,y+ze^j)-\p_{x_j}v_j(s,y)\Big)dydsdz\\
&&\quad=\sum_{j=1}^3\frac{1}{\tau h^3}\frac{1}{h}\int^{0}_{-h}\int_{t_n}^{t_{n+1}}\int_{C_h(x)} \Big(\p_{x_j}v_j(s,y+ze^j)-\p_{x_j}v_j(s,y)\Big)dydsdz,\\
&&|D^-\cdot u^n(x)|^2\le 
\Big(\sum_{j=1}^3\frac{1}{\tau h^3}\frac{1}{h}\int^{0}_{-h}\int_{t_n}^{t_{n+1}}\int_{C_h(x)} |\p_{x_j}v_j(s,y+ze^j)-\p_{x_j}v_j(s,y)|dydsdz\Big)^2\\
&&\quad \le \Big\{\sum_{j=1}^3\Big(\frac{1}{\tau h^3}\frac{1}{h}\Big)\Big(\int^{0}_{-h}\int_{t_n}^{t_{n+1}}\int_{C_h(x)} |\p_{x_j}v_j(s,y+ze^j)-\p_{x_j}v_j(s,y)|^2dydsdz\Big)^\frac{1}{2}\Big\}^2\\
&&\quad\le 3\sum_{j=1}^3\frac{1}{\tau h^3}\frac{1}{h}\int^{0}_{-h}\int_{t_n}^{t_{n+1}}\int_{C_h(x)} |\p_{x_j}v_j(s,y+ze^j)-\p_{x_j}v_j(s,y)|^2dydsdz.
\end{eqnarray*}
Hence, we see that 
\begin{eqnarray*}
&&\sum_{n=0}^{T_\tau-1}\norm D^-\cdot u^n(x)\norm_{2,\Omega_h}^2\tau= \sum_{n=1}^{T_\tau-1}\sum_{x\in\Omega_h} |D^-\cdot u^n(x)|^2h^3\tau\\
&&\quad \le  3\sum_{j=1}^3\frac{1}{h}\int^{0}_{-h}\int_{0}^T\int_{\Omega} |\p_{x_j}v_j(s,y+ze^j)-\p_{x_j}v_j(s,y)|^2dydsdz\\
&&\quad\le  3\sum_{j=1}^3\sup_{0\le z\le h}\norm \p_{x_j}v_j(\cdot,\cdot+ze^j)-\p_{x_j}v_j(\cdot,\cdot)\norm_{L^1([0,T];L^2(\Omega))}^2\quad \to0 \mbox{\quad as $\delta\to0$},\\
&&\sum_{n=0}^{T_\tau-1}|{\rm(i)}|\tau\le T^\frac{1}{2}\Big(\sum_{n=0}^{T_\tau-1}|{\rm(i)}|^2\tau\Big)^\frac{1}{2} 
\le M_2\Big(\sum_{n=0}^{T_\tau-1}\norm D^-\cdot u^n(x)\norm_{2,\Omega_h}^2\tau\Big)^\frac{1}{2}
\to0 \mbox{\quad as $\delta\to0$}.
\end{eqnarray*}
Similarly, we have 
\begin{eqnarray*}
&&|{\rm(ii)}|\le M_1\norm g^{n+1}\norm_{2,\Omega_h}\Big(\sum_{x\in\Omega_h} |u^n(x+he^j)-u^n(x)|^2h^3   \Big)^\frac{1}{2}\\
&&\quad\le M_1\norm f^0\norm_{L^2(\Omega)}
\Big\{\sum_{x\in\Omega_h}    
\Big(\frac{1}{\tau h^3}\int_{t_n}^{t_{n+1}}\int_{C_h(x)} |v_j(s,y+he^j)-v_j(s,y)|dyds\Big)^2h^3
  \Big\}^\frac{1}{2}\\
&&\quad\le M_1\norm f^0\norm_{L^2(\Omega)}
\Big(    
\frac{1}{\tau}\int_{t_n}^{t_{n+1}}\int_{\Omega} |v_j(s,y+he^j)-v_j(s,y)|^2dyds
  \Big)^\frac{1}{2},\\
&&\sum_{n=0}^{T_\tau-1}|{\rm(ii)}|\tau \le  T^\frac{1}{2}\Big(\sum_{n=0}^{T_\tau-1}|{\rm(ii)}|^2\tau\Big)^\frac{1}{2}
\le M_3 \norm v_j(\cdot,\cdot+he^j)-v_j(\cdot,\cdot)\norm_{L^2([0,T];L^2(\Omega))}\\
&&\quad \to0 \mbox{\quad as $\delta\to0$}.
\end{eqnarray*}
Finally, we have  
\begin{eqnarray*}
&&|{\rm(iii)}|\le M_4 \norm u^n_j\norm_{2,\Omega_h}h
\le M_4h\Big\{\sum_{x\in\Omega_h} \Big( \frac{1}{\tau h^3} \int_{t_n}^{t_{n+1}}\int_{C_h(x)} |v_j(s,y)|dyds  \Big)^2h^3\Big\}^\frac{1}{2}\\
&& \le M_4 h\Big( \frac{1}{\tau} \int_{t_n}^{t_{n+1}}\int_{\Omega} |v_j(s,y)|^2dyds  \Big)^\frac{1}{2},\\
&&\sum_{n=0}^{T_\tau-1}|{\rm(iii)}|\tau \le  T^\frac{1}{2}\Big(\sum_{n=0}^{T_\tau-1}|{\rm(iii)}|^2\tau\Big)^\frac{1}{2}
\le M_5h \norm v_j\norm_{L^2([0,T];L^2(\Omega))}\\
&& \to0 \mbox{\quad as $\delta\to0$}.
\end{eqnarray*}
Therefore, applying the same reasoning to $(w^n_j(\cdot)g^{n+1}(\cdot),D_j^+\varphi^{n+1}(\cdot)  )_{2,\Omega_h}$, we see that 
$${\rm (c)}=\sum_{j=1}^3\sum_{n=0}^{T_\tau-1}  \Big(u^n_j(\cdot)g^{n+1}(\cdot),D_j^+\varphi^{n+1}(\cdot)  \Big)_{2,\Omega_h} \tau+ R_\delta,\quad R_\delta\to0 \mbox{\quad as $\delta\to0$}.
$$
Observe that 
\begin{eqnarray*}
&&\sum_{n=0}^{T_\tau-1}  \Big(u^n_j(\cdot)g^{n+1}(\cdot),D_j^+\varphi^{n+1}(\cdot)  \Big)_{2,\Omega_h} \tau\\
&&\quad =\sum_{n=0}^{T_\tau-1} \sum_{x\in\Omega_h} \int_{t_n}^{t_{n+1}}\int_{C_h(x)} v_j(s,y)dyds\times g^{n+1}(x)D_j^+\varphi^{n+1}(x) \\
&&\quad =\sum_{n=0}^{T_\tau-1} \sum_{x\in\Omega_h} \int_{t_n}^{t_{n+1}}\int_{C_h(x)} v_j(s,y) g_\delta(s,y) \p_{x_j}\varphi(s,y)dyds \\
&&\qquad +\sum_{n=0}^{T_\tau-1} \sum_{x\in\Omega_h} \int_{t_n}^{t_{n+1}}\int_{C_h(x)} v_j(s,y) g_\delta(s,y) O(h) dyds \\
&&\quad =\int_0^T\int_{\Omega}v_j(s,y) g_\delta(s,y) \p_{x_j}\varphi(s,y)dyds+O(h).
\end{eqnarray*}
Thus, the weak convergence of $\{g_\delta\}$ implies that 
$${\rm (c)}\to \int_0^T\int_{\Omega}v(s,y) f(s,y) \cdot \nabla\varphi(s,y)dyds\mbox{\quad as $\delta\to0$},$$ 
which conclude that $f$ satisfies \eqref{weak-sol} to be the weak solution of \eqref{T.eq}. 

The property of weak convergence,  (F1) and \eqref{L2-estimate} imply that 
\begin{eqnarray*}
&&T^\frac{1}{2}\norm f^0\norm_{L^2(\Omega)}=\norm f \norm_{L^2([0,T];L^2(\Omega))}\le \liminf_{\delta\to0} \norm g_\delta\norm_{L^2([0,T];L^2(\Omega))},\\
 && \limsup_{\delta\to0} \norm g_\delta\norm_{L^2([0,T];L^2(\Omega))}\le T^\frac{1}{2}\norm f^0\norm_{L^2(\Omega)}.
\end{eqnarray*}
Therefore, the above weak convergence is in fact strong convergence. The uniqueness of $f$ verifies the whole sequence convergence.    
\end{proof}
\setcounter{section}{4}
\setcounter{equation}{0}
\section{Application to the level-set method}

We apply our explicit scheme discussed in Section 3, assuming in addition that the velocity field $v$ belongs to $ C^4([0,T]\times\bar{\Omega};\R^3)$  and initial data $f^0$ belongs to $ C^4(\bar{\Omega};\R)$ (in other words,  the $0$-extensions $\tilde{v}$ and $\tilde{f}_0$ are assumed to be $C^4$-smooth). In actual situations, $v$ would be first given as an element of $C^4([0,T]\times\overline{\Omega_0})$ with $v|_{\p\Omega_0}=0$ and $\Omega_0\subset\R^3$ open, where supp$(v)$ is not necessarily compact; the $0$-extension of such a function would lose smoothness at $\p\Omega_0$; hence, one needs to consider a $C^4$-smooth extension of $v$ in such a way that $v\equiv0$ for all $x\in\R^3\setminus \Omega$ with some bounded open set $\Omega$ containing $\overline{\Omega_0}$ (the same to $f^0$). In the current paper, we do not discuss how to construct such $C^4$-extensions.   

In the smooth case, \eqref{T.eq} is analyzed by the flow $X=X(s,\tau,\xi)$ of the ODE
\begin{eqnarray}\label{5ODE}
x'(s)=v(s,x(s)),\quad x(\tau)=\xi,\quad \tau\in[0,T],\,\,\,\xi\in \bar{\Omega},
\end{eqnarray}
where $X(s,\tau,\xi):=x(s)$ for each $(s,\tau,\xi)\in [0,T]\times[0,T]\times\bar{\Omega}$ and $X$ is $C^4$-smooth in $(s,\tau,\xi)$. By the method of characteristics, the solution $f$ of \eqref{T.eq} is given as 
$$f(t,x)=f^0(X(0,t,x)):[0,T]\times\Omega\to\R,$$
which is $C^4$-smooth and $f\equiv0$ in a neighborhood of $\p\Omega$. The whole space problem
\begin{eqnarray*}
&&x'(s)=\tilde{v}(s,x(s)),\quad x(\tau)=\xi,\quad \tau\in[0,T],\,\,\,\xi\in \R^3
\end{eqnarray*}
has the $C^4$-smooth flow $\tilde{X}$ such that 
$\tilde{X}(s,\tau,\xi)=X(s,\tau,\xi)$ if $\xi\in\bar{\Omega}$, and $\tilde{X}(s,\tau,\xi)\equiv\xi$ if $\xi\not\in\Omega$.  
 The whole space problem 
\begin{eqnarray*}
&&\p_t\tilde{f}+\tilde{v}\cdot\nabla\tilde{f}=0\mbox{ \,\,\,in $[0,T]\times\R^3$,\quad $\tilde{f}(0,\cdot)=\tilde{f}^0$ on $\R^3$} 
\end{eqnarray*}
has the unique $C^4$-solution $\tilde{f}(t,x)=\tilde{f}^0(\tilde{X}(0,t,x))$, which is the $0$-extension of the above $f$ outside $\bar{\Omega}$. In the rest of this section, we do not distinguish $v,f^0,f,X$ and $\tilde{v},\tilde{f}^0,\tilde{f},\tilde{X}$.     

We first demonstrate error estimates for  our explicit scheme up to the second order $x$-derivatives of $f$, where  $C^4$-smoothness is necessary for better error bounds. Then, we apply the results to approximate the level-set defined by 
$$\Gamma(t):=\{ x\in\Omega \,|\,f(t,x)=1\},\quad t\ge0$$ 
with additional assumptions on $f^0$. The error estimates  up to the second order $x$-derivatives of $f$ yield good  approximation of the unit normal vectors,  the mean curvature and the area element of $\Gamma(t)$. 

Throughout this section, the truncation argument with $\tilde{u}^n$ given in \eqref{truncation} and the generalized hyperbolic scaling are not necessary, since $v$ is a priori bounded, i.e., one can take the standard hyperbolic scale $\tau=\frac{2}{7}\norm v\norm_{L^\infty([0,T]\times\Omega)}^{-1} h$. Nevertheless, we argue under the generalized hyperbolic scale condition $\tau=O(h^{2-\alpha})$ with any $0<\alpha\le1$ to include the case  \eqref{scale} ($\alpha=1$ is not essential).  Below, ``$(h,\tau)\to0$'' always means $(h,\tau)\to0$ under this scale condition.  

Hereafter, $O(r)$ stands for a quantity whose absolute value is bounded by $M |r|$ with some constant $M\ge0$ independent of $x,n$ and $r\to0$; $M_1,M_2,\ldots$ stand for some positive constants independent of $\delta=(h,\tau)$, $x$ and $n$.
 
\subsection{Error estimate for explicit scheme (smooth case)} 

The argument below inductively shows that the explicit scheme \eqref{naive} and the difference equations derived from the discrete $x$-differentiation of \eqref{naive}  approximate $f$ and its $x$-derivatives up to the $k$-th order ($k\ge0$) with the error bound $O(h^\alpha)$, provided $f$ is $C^{k+2}$-smooth. We only demonstrate the case of $k=2$. It is essential that the explicit scheme does not involve any boundary condition in space.      

By Taylor's approximation and \eqref{T.eq}, we have for any $x\in h\Z^3$ and $0\le n\le T_\tau-1$, 
\begin{eqnarray*}
f(t_{n+1},x)=\frac{1}{7}\sum_{\omega\in B}f(t_n,x+h \omega)-u^n(x)\cdot Df(t_n,x)\tau+O(\tau^2)+O(h^2)+O(\tau h), 
\end{eqnarray*}
where $O(\tau^2)+O(h^2)+O(\tau h)=O(h^\alpha\tau )$. 
Let $g^n(\cdot)$ be the solution of the explicit problem \eqref{naive} and set $b^n(x):=g^n(x)-f(t_n,x)$. Then, we have  $b^0(x)=O(h)$ on $h\Z^3$ and 
\begin{eqnarray*}
b^{n+1}(x)&=&\frac{1}{7}\sum_{\omega\in B}b^n(x+h \omega)-u^n(x)\cdot Db^n(t_{n,x})\tau+O(h^\alpha \tau )\mbox{\quad on $h\Z^3$}.
\end{eqnarray*}
The CFL-condition \eqref{3CFL} implies that
\begin{eqnarray}\label{error1}
&& \max_{x\in h\Z^3}|b^{n+1}(x)|\le  \max_{x\in h\Z^3}|b^{n}(x)|+M_1h^\alpha\tau\\\nonumber
&&\quad \le \max_{x\in h\Z^3}|b^{n-1}(x)|+2M_1h^\alpha\tau\\\nonumber
&&\quad\le \cdots\le \max_{x\in h\Z^3}|b^{0}(x)|+M_1T h^\alpha \le M_2h^\alpha,\quad 0\le \forall\,n\le T_\tau-1,\mbox{\quad  $\forall\,(h,\tau)\to0$}. 
\end{eqnarray} 
\indent Observe that $f_i:=\p_{x_i }f$ and  $g^n_i:=D^+_ig^n$ ($i=1,2,3$) satisfy
\begin{eqnarray}\label{fi}
&&\p_tf_i+v\cdot \nabla f_i+\p_{x_i}v\cdot \nabla f=0\mbox{\quad in $[0,T]\times\R^3$},\\\nonumber
&&g^{n+1}_i(x)=\frac{1}{7}\sum_{\omega\in B}g^n_i(x+h \omega)-u^n(x+he_i)\cdot Dg^n_i(x)\tau\\\nonumber
&&\qquad +\frac{1}{2}\sum_{j=1}^3D^+_iu^n_j(x)\Big(g^n_j(x)+g^n_j(x-he^j)\Big)\tau\quad\mbox{ on $h\Z^3$},
\end{eqnarray}
respectively. By Taylor's approximation and \eqref{fi}, we have 
\begin{eqnarray*}
&&f_i(t_{n+1},x)=\frac{1}{7}\sum_{\omega\in B}f_i(t_n,x+h \omega)-u^n(x+he_i)\cdot Df_i(t_n,x)\tau\\\nonumber
&&\qquad +\frac{1}{2}\sum_{j=1}^3D^+_iu^n_j(x)\Big(f_j(t_n,x)+f_j(t_n,x-he^j)\Big)\tau+O(h^\alpha\tau).
\end{eqnarray*}
Hence, $b^n_i(x):=g^n_i(x)-f_i(t_n,x)$ and $R^n:=\max_{i=1,2,3,\,\,x\in h\Z^3}|b^n_i(x)|$ satisfy  
\begin{eqnarray*}
b^{n+1}_i(x)&=&\frac{1}{7}\sum_{\omega\in B}b^n_i(x+h \omega)-u^n(x+he_i)\cdot Db^n_i(t_{n,x})\tau\\
&&\quad +\frac{1}{2}\sum_{j=1}^3D^+_iu^n_j(x)\Big(b^n_j(x)+b^n_j(x-he^j)\Big)\tau+O(h^\alpha\tau),\\
|b^{n+1}_i(x)|&\le& R^n
+\max_{j=1,2,3,\,\,x\in h\Z^3}|D^+_iu^n_j(x)| R^n\tau +M_3 h^\alpha\tau,\quad \forall\,x\in h\Z^3, \\
R^{n+1}&\le&   (1+M_4\tau) R^n+M_3 h^\alpha\tau.     
\end{eqnarray*}
Therefore,  noting that $R^0=O(h)$, we obtain 
\begin{eqnarray}\label{error2}
\max_{x\in h\Z^3}|b^n_i(x)|&\le& R^n\le e^{M_4 T}\Big(R^0+\frac{M_3}{M_4}h^\alpha\Big )\\\nonumber
& \le& M_5 h^\alpha,\quad  0\le \forall\,n \le T_\tau, \mbox{\quad  $\forall\,(h,\tau)\to0$}.
\end{eqnarray}
\indent Observe that $f_{ij}:=\p_{x_j }\p_{x_i }f$,  $g^n_{ij}:=D^-_j(D^+_ig^n$) ($i,j=1,2,3$) satisfy
\begin{eqnarray}\label{fij}
&&\p_tf_{ij}+v\cdot \nabla f_{ij}+\p_{x_j}v\cdot \nabla f_{i}+\p_{x_i}v\cdot \nabla f_j+\p_{x_j}\p_{x_i}v\cdot \nabla f=0\mbox{\quad in $[0,T]\times\R^3$},\\\nonumber
&&g^{n+1}_{ij}(x)=\frac{1}{7}\sum_{\omega\in B}g^n_{ij}(x+h \omega)-u^n(x+he_i)\cdot Dg^n_{ij}(x)\tau\\\nonumber
&&\qquad +\frac{1}{2}\sum_{k=1}^3D^-_ju^n_k(x+he^j)\Big(g^n_{ik}(x-he^j+he^k)+g^n_{ik}(x-he^j)\Big)\tau\\\nonumber
&&\qquad +\frac{1}{2}\sum_{k=1}^3D^+_iu^n_k(x)\Big(g^n_{jk}(x-he^j+he^k)+g^n_{jk}(x-he^j)\Big)\tau\\\nonumber
&&\qquad +\frac{1}{2}\sum_{k=1}^3D^-_jD^+_iu^n_k(x)\Big(g^n_{k}(x-he^j)+g^n_{k}(x-he^j-he^k)\Big)\tau\mbox{\quad on $h\Z^3$},
\end{eqnarray}
respectively. By Taylor's approximation and \eqref{fij}, we have 
\begin{eqnarray*}
&&f_{ij}(t_{n+1},x)=\frac{1}{7}\sum_{\omega\in B}f_{ij}(t_n,x+h \omega)-u^n(x+he_i)\cdot Df_{ij}(t_n,x)\tau\\\nonumber
&&\qquad +\frac{1}{2}\sum_{k=1}^3D^-_ju^n_k(x+he^j)\Big(f_{ik}(t_n,x-he^j+he^k)+f_{ik}(t_n,x-he^j)\Big)\tau\\\nonumber
&&\qquad +\frac{1}{2}\sum_{k=1}^3D^+_iu^n_k(x)\Big(f_{jk}(t_n,x-he^j+he^k)+f_{jk}(t_n,x-he^j)\Big)\tau\\\nonumber
&&\qquad +\frac{1}{2}\sum_{k=1}^3D^-_jD^+_iu^n_k(x)\Big(f_{k}(t_n,x-he^j)+f_{k}(t_n,x-he^j-he^k)\Big)\tau+O(h^\alpha\tau).
\end{eqnarray*}
Set $b^n_{ij}(x):=g^n_{ij}(x)-f_{ij}(t_n,x)$ and $\tilde{R}^n:=\max_{i,j=1,2,3,\,\,x\in h\Z^3}|b^n_{ij}(x)|$. Since we already know that $b^n_k(x)=O(h^\alpha)$ for $k=1,2,3$, we have   
\begin{eqnarray*}
|b^{n+1}_{ij}(x)|\le \tilde{R}^{n+1}\le \tilde{R}^n+M_6\tau \tilde{R}^n+M_7h^\alpha \tau.
\end{eqnarray*}
Therefore,  we obtain 
\begin{eqnarray}\label{error3}
\max_{x\in h\Z^3}|b^n_{ij}(x)|\le M_8 h^\alpha,\quad  0\le \forall\,n \le T_\tau,\mbox{\quad $\forall\,(h,\tau)\to0$}.
\end{eqnarray}

\subsection{Approximation of smooth level-set} 

Suppose in addition that  initial data $f^0$ satisfies the following conditions:  
\begin{itemize}
\item There exists a connected open set $\Omega^+$ such that $\overline{\Omega^+}\subset\Omega$, $f^0(x)>1$ in $\Omega^+$ and $f^0(x)<1 $ in $\Omega^-:=\Omega\setminus \overline{\Omega^+}$,
\item $\nabla f_0(x)\neq0$ on $\Gamma(0):=\p\Omega^+$.
\end{itemize}
It follows from the representation $f(t,x)=f^0(X(0,t,x))$ that 
\begin{eqnarray*}
&&\Omega^+(t):=\{x\in\Omega\,|\,f(t,x)>1\}=\{x\in\Omega\,|\,X(0,t,x)\in\Omega^+\},\\
&&\Gamma(t):=\{x\in\Omega\,|\,f(t,x)=1\}=\{ x\in\Omega\,|\, X(0,t,x)\in\Gamma(0)\}=\p\Omega^+(t)=X(t,0,\Gamma(0)).
\end{eqnarray*}
Hence, the identity $X(t,0,X(0,t,x))\equiv x$ implies that 
$$\nabla f(t,x)=\nabla f^0(X(0,t,x))\p_\xi X(0,t,x)\neq 0,\,\,\,\,\forall\,x\in\Gamma(t),\,\,\,\forall\,t\in[0,T].$$
 Therefore, $\Gamma(t)$ is a connected closed $C^4$-surface for each $t\in[0,T]$.    
Define the set for $\mu_0\ge0$, 
$$ \Theta_{\mu_0}:=\bigcup_{0\le t\le T} \{t\}\times(\Gamma(t)+\mu_0), \,\,\,\,\Gamma(t)+\mu_0:=\{ y\in\Omega\,|\,|y-x|\le \mu_0,\,\,\,x\in\Gamma(t) \}.$$
 There exist constants $\mu_0>0$ and  $0<C_1\le C_2$  such that $ \Theta_{\mu_0}\subset[0,T]\times\Omega$ and 
 \begin{eqnarray}\label{5grad}
C_1\le |\nabla f(t,x)|\le C_2,\quad\forall\,(t,x)\in \Theta_{\mu_0}.
\end{eqnarray}
If necessary, we re-take $\mu_0>0$ small enough so that the intersection of $\Gamma(t)$ and the $\mu_0$-ball with the center  $x\in \Gamma(t)$ is ``almost flat'' for all $x\in \Gamma(t)$. 
Define the unit vector field $\nu(t,x)$ and a function $m(t,x)$ as  
 $$\nu(t,x):=\frac{\nabla f(t,x)}{|\nabla f(t,x)|}:\Theta_{\mu_0}\to \R^3,\quad m(t,x):=-\nabla\cdot \nu(t,x):\Theta_{\mu_0}\to\R.$$   
Note that {\it the restriction of $\nu$ on $\{t\}\times\Gamma(t)$ is a unit normal of $\Gamma(t)$; the value $m(t,x)$ of $x\in\Gamma(t)$ is the mean curvature of $\Gamma(t)$ at $x$ in the direction of  $\nu(t,x)$;  
$$m(t,x)=-\frac{1}{|\nabla f(t,x)|}\Big(\Delta f(t,x)-\sum_{i,j=1}^3 \p_{x_i}\p_{x_j}f(t,x)\nu_{i}(t,x)\nu_{j}(t,x)\Big),$$
where $\Delta=\p_{x_1}^2+\p_{x_2}^2+\p_{x_3}^2$. }
From now on, we suppose that $\nu(t,\cdot)|_{\Gamma(t)}$ is the outer unit normal of $\Gamma(t)$ with respect to $\Omega^+(t)$ (otherwise, we re-define $\nu$ as $-\nabla f(t,x)/|\nabla f(t,x)|$). 
For each $\tilde{x}\in \Gamma(t)$, at least one of $\p_{x_i}f(\tilde{x})$, $i=1,2,3$ is away from $0$, say, $i=3$; then, $ \Gamma(t)$ is represented in a neighborhood of $\tilde{x}$ by the hight function as $x_3=\varphi(t,x_1,x_2)$, $(x_1,x_2)\in O$ with some open set $O\subset \R^2$;  the area element $dS(x)$ of  $ \Gamma(t)$ is represented in the  neighborhood of $\tilde{x}$ as 
\begin{eqnarray*}
dS(x)&=&\sqrt{\p_{x_1}\varphi(t,x_1,x_2)^2+\p_{x_2}\varphi(t,x_1,x_2)^2+1}\,\,dx_1dx_2\\
&=&\frac{1}{|\p_{x_3}f(t,x)|}\sqrt{\p_{x_1}f(t,x)^2+\p_{x_2}f(t,x)^2+\p_{x_3}f(t,x)^2}\,\,dx_1dx_2,
\end{eqnarray*}
which is also represented, if $\p_{x_i}f(t,x)\neq0$, as  
\begin{eqnarray*}
dS(x)&=&\frac{1}{|\p_{x_i}f(t,x)|}\sqrt{\p_{x_j}f(t,x)^2+\p_{x_k}f(t,x)^2+\p_{x_i}f(t,x)^2}\,\,dx_jdx_k
\end{eqnarray*}
with $\{i,j,k\}=\{1,2,3\}$.
We refer to \cite{Giga} and \cite{pruss} for more details on differential geometry of moving surfaces.

Now, we introduce discrete objects corresponding to the above. Define 
\begin{eqnarray*}
&&\Omega^n_+:=\{ x\in\Omega_h \,|\,g^n(x)>1 \},\quad 
\bar{\Omega}^n_+:=\{  x+h\omega \,|\,x\in \Omega^n_+,\,\,\,\omega\in B\},\\
&&\Gamma^n:=\p\bar{\Omega}^n_+.
\end{eqnarray*}
$\Gamma^n$ and $\Gamma^{n+1}$ are ``connected'' in the following sense:  {\it for each $y\in\Gamma^{n+1}$ (then, $g^{n+1}(y)\le 1$ and $g^{n+1}(y+h\omega)> 1$ for some $\omega\in B$ by definition), the set $K(y):=\{ y+h\omega\,|\,\omega \in B\}$ contains at least one element of $\Gamma^n$}. In fact, $K(y)\subset \Omega^n_+$ is not possible,  since the CFL-condition implies $g^{n+1}(y)>1$; if $K(y)\cap \Gamma^n=\emptyset$, we must have $g^n(y+h\omega)\le 1$ and  $g^n(y+h\omega+h\omega')\le 1$ for any $\omega,\omega'\in B$; then, due to the CFL condition, we necessarily have $g^{n+1}(y+h\omega)\le1$ for all $\omega\in B$, which is not allowed.  

Define for $A,A'\subset\R^3$,
 $${\rm dist}(A,A'):=\sup_{y\in A}{\rm dist}(y,A'),\quad {\rm dist}(y,A'):=\inf_{x\in A'}|y-x|,$$  
where we have ${\rm dist}(A,A')=0\Leftrightarrow \overline{A}\subset\overline{A'}$. 
\begin{Thm}\label{level1}
There exists a constant $M_9>0$ independent from $(h, \tau)$ such that  
\begin{eqnarray*}
\max_{0\le n\le T_\tau}{\rm dist}(\Gamma^n,\Gamma(t_n))\le M_9h^\alpha\mbox{\quad for all $(h,\tau)\to0$}.
\end{eqnarray*}
\end{Thm}
\begin{proof}
Since $f\equiv0$ in a neighborhood of $\p\Omega$, \eqref{error1} implies that $\Gamma^n$ is far away from $\p\Omega_h$ for all sufficiently small $h>0$. 
Fix any $y\in\Gamma^n$.  We have $g^n(y)\le 1$ and $g^n(y+h\omega)>1$ for some $\omega\in B$. The estimate \eqref{error2} implies that $|g^n(\tilde{y})-g^n(\tilde{y}+h\omega)|\le M_{10}h$ for any $\tilde{y}$ and $n$. Hence, with \eqref{error1}, we have  
\begin{eqnarray*}
&& 1-M_{10}h<g^n(y+h\omega)-M_{10}h\le g^n(y),\\
&&-M_2 h^\alpha\le f(t_n,y)-g^n(y)\le f(t_n,y)-(1-M_{10}h),\\
&& f(t_n,y)-1\le f(t_n,y)-g^n(y)\le M_2h^\alpha,\\
&&| f(t_n,y)-1|\le M_{11}h^\alpha.
\end{eqnarray*}
Since $|f(t,x)-1|>0$ for all $x\not\in\Gamma(t)$ and $t\in[0,T]$, there exists a constant $\ep_0>0$ such that 
$$(t,x)\in[0,T]\times\Omega,\,\,\,\,|f(t,x)-1|<\ep_0\Rightarrow (t,x)\in \Theta_{\mu_0}.$$ 
 If $h>0$ is small enough to satisfy $ M_{11}h^\alpha< \ep_0$, we have $(t_n,y)\in \Theta_{\mu_0}$; furthermore, there exist a point $y^\ast=y^\ast(y)\in\Gamma(t_n)$ and a number $\mu\in[-\mu_0,\mu_0]$ such that $y-y^\ast=\mu \nu(t_n,y^\ast)$.  
Therefore, it follows from \eqref{5grad} that 
\begin{eqnarray*}
 M_{11}h^\alpha \ge | f(t_n,y)-1|=| f(t_n,y)-f(t_n,y^\ast)|\ge C_1|y-y^\ast|-M_{12}|y-y^\ast|^2,
\end{eqnarray*}
which leads to $|y-y^\ast|\le M_{13}h^\alpha$ (we must choose $0<\mu_0\ll C_1$ for this inequality). Since dist$(y,\Gamma(t_n))\le M_{13}h^\alpha$ and $y\in\Gamma^n$ is arbitrary, we obtain the assertion. 
\end{proof}
Theorem \ref{level1} and \eqref{error2} imply that $D^+g^n\neq0$ on $\Gamma^n$ for all sufficiently small $h>0$. Hence, the following quantities are well-defined: 
\begin{eqnarray*}
&&\nu^n(y):=\frac{D^+g^n(y)}{|D^+g^n(y)|},\quad y\in\Gamma^n,\\
&&m^n(y):=-\frac{1}{|D^+g^n(y)|}\Big( \sum_{i=1}^3D^2_ig^n(y)-\sum_{i,j=1}^3D^-_iD^+_j g^n(y) \nu^n_i(y)\nu^n_j(y) \Big),\quad y\in\Gamma^n.
\end{eqnarray*}
Let $y^\ast=y^\ast(y)$ with $y\in\Gamma^n$ be the point of $\Gamma(t_n)$ given in the proof of Theorem \ref{level1}, i.e.,  
\begin{eqnarray}\label{normal}
y^\ast-y=\mu \nu(t_n,y^\ast),\,\,\,\, |y^\ast-y|=|\mu|\le M_{13}h^\alpha,
\end{eqnarray}
where $M_{13}$ is independent from $(h,\tau)$, $n$ and $y\in\Gamma^n$.   
\begin{Thm}\label{level2}
There exists constants $M_{14},M_{15}>0$ independent from $(h,\tau)$ such that 
\begin{eqnarray*}
&&\max_{y\in\Gamma^n,\,0\le n\le T_\tau}|\nu^n(y)-\nu(t_n,y^\ast)|\le M_{14}h^\alpha, \\
&&\max_{y\in\Gamma^n,\,0\le n\le T_\tau}|m^n(y)-m(t_n,y^\ast)|\le M_{15}h^\alpha\quad\mbox{ for all $(h,\tau)\to0$}.
\end{eqnarray*}
\end{Thm}
\begin{proof}
Since $\nu$ and $m$ are $C^1$-smooth on $\Theta_{\mu_0}$, the assertion follows from \eqref{error2}, \eqref{error3}, \eqref{5grad} and \eqref{normal}.
\end{proof}
Finally, we give the notion of the the ``area element'' and ``surface area'' of $\Gamma^n$ that approximates $dS(x)$ and $|\Gamma(t)|$. For each $y\in\Gamma^n$, define the area element $dS^n(y)$ at $y$ as 
\begin{eqnarray*}
&&dS^n(y):=\frac{1}{|D^+_ig^n(y)|}\sqrt{D^+_jg^n(y)^2+D^+_kg^n(y)^2+D^+_ig^n(y)^2}\,\,h^2,\\
&&\{i,j,k\}=\{1,2,3\},\quad\mbox{$i$ is such that $\dis |D^+_ig^n(y)|:=\max_{l=1,2,3}|D^+_lg^n(y)|$},
\end{eqnarray*}  
where $D^+_ig^n(y)$ is always away from $0$ due to \eqref{error2} and \eqref{5grad}. In order to connect calculus with $dS^n$ on $\Gamma^n$ and the surface integral on $\Gamma(t_n)$, we introduce the (local)  ``hight functions'' of $\Gamma^n$ in terms of the direction of $e^1$, $e^2$ or $e^3$.  $\Gamma^n$ might contain unnecessary points that  make the ``hight functions'' multi-valued, and we get rid of these points from $\Gamma^n$ through the following refinement process. 

 It follows from \eqref{error2} and \eqref{5grad} that 
$\frac{C_1}{2}\le |D^+ g^n|\le 2C_2$ on $\Gamma^n$ for all $n$ and all sufficiently small $h>0$.  We see that for each $y\in \Gamma^n$,  
$$\frac{C_1}{6}\le \max_{j=1,2,3}|D^+_j g^n(y)|.$$
 Due to \eqref{error3}, there exists a constant $M_{16}>0$ such that $|D^+_j g^n(x+h\omega)-D^+_j g^n(x)|\le M_{16}h$  for all $x\in h\Z^3$, $\omega\in B$, $j=1,2,3$ and $0\le n\le T_\tau$. Suppose that $y^{(1)}\in\Gamma^n$ is such that  $\frac{C_1}{6}\le |D^+_i g^n(y^{(1)})|$. Then, we have 
\begin{eqnarray}\label{5555}
 &&\frac{C_1}{8}\le |D^+_i g^n(\tilde{y})|\quad \mbox{for all $\tilde{y}\in h\Z^3\cap O_\ep(y^{(1)})$},\\\nonumber
 &&O_{\ep_1}(y^{(1)}):=\Big\{ x\in\R^3\,|\, |x_i-y^{(1)}_i|<\ep_1,\,\,\, |x_j-y^{(1)}_j|<\frac{1}{2}a\ep_1,\,\,\,j\in\{1,2,3 \}\setminus\{i\} \Big\},\\\nonumber 
 && \ep_1:= \min\Big\{\frac{C_1}{100 M_{16}},\frac{\mu_0}{3}\Big\},\,\,\,\\\nonumber
 &&a:=\min_{\frac{C_1}{12 C_2}\le w<1} \frac{w}{\sqrt{1-w^2}}\,\,\,\,\,\,\,\,\,\,\Big(\mbox{note that } a\le\min\Big\{1, \frac{|\nu^n(y^{(1)})\cdot e^i|}{\sqrt{1-(\nu^n(y^{(1)})\cdot e^i)^2}}\Big\}\Big). 
\end{eqnarray}
Let $P(y^{(1)})$ be the plane in $\R^3$ containing $y^{(1)}$ whose normal direction is $\nu^n(y^{(1)})$. Let $P_\ep (y^{(1)})$ be the union with respect to $x\in P(y^{(1)})$  of all $\ep$-ball with the center $x$. 
Theorem \ref{level1} implies that for 
$$\ep := \frac{1}{4}\cdot\frac{C_1}{12C_2}\cdot\ep_1\,\,\,\,\,\,\,\,\,\,\Big(\mbox{note that }\ep\le   \frac{1}{4}|\nu^n(y^{(1)})\cdot e^i|\ep_1\Big)$$
 and all $0<h\ll \ep$, we have 
$$\Xi^{(1)}:=\Gamma^n\cap O_{\ep_1}(y^{(1)})\subset P_\ep(y^{(1)})\cap O_{\ep_1}(y^{(1)}).$$  
For $x=(x_1,x_2,x_3)\in A\subset\R^3$, the point $(x_j,x_k)$, $j<k$ is said to be the  $e^i$-projection of $x$ and denoted by $\Pi_{e^i}x$; the family of all $e^i$-projected points of $A$ is said to be the  $e^i$-projection of $A$ and denoted by $\Pi_{e^i}A$.
We proceed in the case of $i=3$ in \eqref{5555}  (the other cases are the same). 
Consider the  line $l(\hat{y}):=\{(\hat{y}_1,\hat{y}_2,z)\,|\,z\in\R\}$ for each $\hat{y}=(\hat{y}_1,\hat{y}_2)\in  \Pi_{e^3}  (h\Z^3\cap O_{\ep_1}(y^{(1)})$.   
For each $\hat{y}$,  the  intersection $l(\hat{y})\cap P_\ep(y^{(1)})\cap O_{\ep_1}(y^{(1)})$ is not empty and it   must contain at least one point of $\Gamma^n$, because $f$ is strictly smaller than $1$  at one of the two end points of $l(\hat{y})\cap P_\ep(y^{(1)})\cap O_{\ep_1}(y^{(1)})$ and   strictly larger than $1$ at the other end point;  \eqref{error1} implies that $g^n-1$ must change the sign within $l(\hat{y})\cap P_\ep(y^{(1)})\cap O_{\ep_1}(y^{(1)})\cap h\Z^3$. If $l(\hat{y})\cap P_\ep(y^{(1)})\cap O_{\ep_1}(y^{(1)})$ contains more than two points of $\Gamma^n$, we take the one $y(\hat{y})$ whose third component is smallest. Define the set 
$$\tilde{\Xi}^{(1)}:=\{ y(\hat{y})\,|\,  \hat{y}\in   \Pi_{e^3}  (h\Z^3\cap O_{\ep_1}(y^{(1)}) \}.$$
Consider a point $y^{(2)}\in \Gamma^n\setminus \Xi^{(1)}$ which is close to $ O_{\ep_1}(y^{(1)})$, and repeat the above procedure with $ \Gamma^n\setminus \Xi^{(1)}$ instead of $\Gamma^n$, where we may use the same constants $\ep_1$ and $\ep$, to obtain $\Xi^{(2)}$ and $\tilde{\Xi}^{(2)}$.  We continue the same procedure with a point $y^{(3)}\in \Gamma^n\setminus (\Xi^{(1)}\cup\Xi^{(2)})$ which is close to $ O_{\ep_1}(y^{(1)})\cup O_{\ep_1}(y^{(2)})$. The process continues until $\Xi^{(1)}\cup\Xi^{(2)}\cup\cdots\cup\Xi^{(K)}$ gets equal to $\Gamma^n$, where $K$ might depend on $y^{(1)}$ or $(h,\tau)$. However,  we have $K\le K^\ast$ with some number $K^\ast\in\N$ independent from $0<h\ll \ep$, because $\Gamma^n$ is bounded. We may construct $ O_{\ep_1}(y^{(1)}),\ldots, O_{\ep_1}(y^{(K)})$ so that they cover a closed surface (i.e., $ O_{\ep_1}(y^{(1)})\cup\cdots\cup O_{\ep_1}(y^{(K)})$ dose not have any ``hole''). 
Let $\tilde{\Gamma}^n$ be the refinement of $\Gamma^n$ defined as 
$$\tilde{\Gamma}^n:=\tilde{\Xi}^{(1)}\cup\tilde{\Xi}^{(2)}\cup\cdots\cup\tilde{\Xi}^{(K)}.$$  
\begin{Thm}\label{level4}
Let $t\in[0,T]$ be arbitrary and $n\in\{0,\ldots,T_\tau\}$ be such that $t\in[t_n,t_{n+1})$ for each $(h,\tau)$. For any continuous function $\phi:\Omega\to\R$, we have 
$$\sum_{y\in\tilde{\Gamma}^n}\phi(y)dS^n(y)-\int_{\Gamma(t_n)}\phi(x)dS(x)\to0\mbox{\quad as $(\tau,h)\to0$}.$$
In particular, if $\phi$ is Lipschitz continuous, the above convergence is of the rate $O(h^\alpha)$.   
\end{Thm}
\begin{proof}
Note that $\{ O_{\ep_1}(y^{(1)}),\ldots, O_{\ep_1}(y^{(K)})\}$ is a covering of $\Gamma(t_n)$;   
  each $O_{\ep_1}(y^{(k)})\cap\Gamma(t_n)$ is represented by a hight function; $\phi$ is uniformly continuous on $ \overline{O_{\ep_1}(y^{(1)})\cup\cdots\cup O_{\ep_1}(y^{(K)})}$.

Suppose that $\Xi^{(1)}$ is given with $\Pi_{e^3}$ (the other cases are the same).  $O_{\ep_1}(y^{(1)})\cap\Gamma(t_n)$ is represented by a hight function $x_{3}=\varphi^{(1)}(x_1,x_2)$ for $(x_1,x_2)\in A^{(1)}:=\Pi_{e^3}O_{\ep_1}(y^{(1)})$. By the already obtained convergence results, we see that 
\begin{eqnarray*}
&&\!\!\!\!\!\!\!\!\!\int_{O_{\ep_1}(y^{(1)})\cap\Gamma(t_n)}\phi(x)dS\\
&&\!\!\!\!\!\!\!\!\!=\int_{A^{(1)}} \phi (x_1,x_2,\varphi^{(1)}(x_1,x_2))\frac{1}{|\p_{x_3}f(t,x)|}\sqrt{\p_{x_1}f(t,x)^2+\p_{x_2}f(t,x)^2+\p_{x_3}f(t,x)^2}\,\,dx_1dx_2\\
&&\!\!\!\!\!\!\!\!\!=\sum_{y\in \tilde{\Xi}^{(1)}}
\phi (y) dS^n(y) + \eta^{(1)}(h,\tau)|A^{(1)}|,
\end{eqnarray*} 
where $\eta^{(1)}(h,\tau)$ is determined by the uniform continuity of $\phi$ and $\eta^{(1)}(h,\tau)=O(h^\alpha)$ if $\phi$ is Lipschitz.  Similarly, we see that 
\begin{eqnarray*}
\int_{(O_{\ep_1}(y^{(2)})\setminus O_{\ep_1}(y^{(1)}))\cap\Gamma(t_n)}\phi(x)dS
=\sum_{y\in \tilde{\Xi}^{(2)}}
\phi (y) dS^n(y) + \eta^{(2)}(h,\tau)|A^{(2)}|,
\end{eqnarray*} 
where  $\eta^{(2)}(h,\tau)$ has the same bound as  $\eta^{(1)}(h,\tau)$ and  $A^{(2)}$ is the $e^i$-projection of $O_{\ep_1}(y^{(2)})\setminus O_{\ep_1}(y^{(1)})$;
\begin{eqnarray*}
\int_{\{O_{\ep_1}(y^{(3)})\setminus (O_{\ep_1}(y^{(1)})\cup O_{\ep_1}(y^{(2)})\}\cap\Gamma(t_n)}\phi(x)dS
=\sum_{y\in \tilde{\Xi}^{(3)}}
\phi (y) dS^n(y) + \eta^{(3)}(h,\tau)|A^{(3)}|,
\end{eqnarray*} 
where   $\eta^{(3)}(h,\tau)$ has the same bound as  $\eta^{(1)}(h,\tau)$ and  $A^{(3)}$ is the $e^{i'}$-projection of $O_{\ep_1}(y^{(3)})\setminus (O_{\ep_1}(y^{(1)})\cup O_{\ep_1}(y^{(2)}))$, and so on. Hence, by summarizing these equalities, we obtain 
\begin{eqnarray*}
\int_{\Gamma(t)}\phi(x) dS=\sum_{y\in\tilde{\Gamma}^n}\phi(y)dS^n(y)+\sum_{k=1}^K\eta^{(k)}(h,\tau)|A^{(k)}|.
\end{eqnarray*}
Since $|A^{(k)}|\le (\frac{1}{2}a\ep_1)^2$ and $K$ is finite, we conclude our assertion.  
\end{proof}       
\begin{Cor}\label{level5}
Let $t\in[0,T]$ be arbitrary and $n\in\{0,\ldots,T_\tau\}$ be such that $t\in[t_n,t_{n+1})$ for each $(h,\tau)$. For any smooth function $\phi:[0,T]\times\Omega\to\R$, we have 
$$\Big|\sum_{y\in\tilde{\Gamma}^n}\phi(t_n,y)dS^n(y)-\int_{\Gamma(t)}\phi(t,x)dS\Big|\le M_{16} h^\alpha \mbox{\quad for all $(\tau,h)\to0$}.$$
\end{Cor}
\begin{proof}
By the transport theorem (see, e.g., \cite{pruss}), we have 
$$\Big|\int_{\Gamma(t_n)}\phi(t_n,x)dS-\int_{\Gamma(t)}\phi(t,x)dS\Big|\le M_{17}\tau.$$  
Then, similar reasoning for Theorem \ref{level4}  leads to the assertion. 
\end{proof}

\medskip\medskip\medskip

\noindent{\bf Acknowledgement.} 
This work was written during author's one-year research stay  in Fachbereich Mathematik, Technische Universit\"at Darmstadt, Germany, with the grant Fukuzawa Fund (Keio Gijuku Fukuzawa Memorial Fund for the Advancement of Education and Research). The author would like to express special thanks to Professor Dieter Bothe for his kind hosting in TU-Darmstadt and to Professor Stefano Modena for his personal discussions on DiPerna-Lions theory.  The author is supported by JSPS Grant-in-aid for Young Scientists \#18K13443 and JSPS Grants-in-Aid for Scientific Research (C) \#22K03391. 

\end{document}